\documentclass{article}
\usepackage{amsmath}
\usepackage{quiver}
\usepackage{amsthm}
\usepackage{amssymb}
\usepackage{amsfonts}
\usepackage{geometry}
\usepackage{graphicx} 
\usepackage[
backend=biber,
style=alphabetic,
sorting=nyt
]{biblatex}
\usepackage{mathrsfs}
\usepackage{xcolor}
\usepackage{hyperref}

\newtheorem{thm}{Theorem}[section]
\newtheorem{prop}[thm]{Proposition}

\newtheorem{lemma}[thm]{Lemma}
\newtheorem{conj}[thm]{Conjecture}

\theoremstyle{definition}
\newtheorem{defn}[thm]{Definition}

\theoremstyle{remark}
\newtheorem*{rmk}{Remark}

\usepackage{thm-restate}

\DeclareMathOperator{\Spec}{Spec}

\DeclareMathOperator{\sll}{SL}

\DeclareMathOperator{\et}{et}
\DeclareMathOperator{\geom}{geom}
\DeclareMathOperator{\piet}{\pi_1^{\et}}
\DeclareMathOperator{\pietgeo}{\pi_1^{\et, \geom}}
\usepackage{graphicx} 

\title{Finiteness for \'{E}tale Fundamental Groups of N\'{e}ron Models}
\author{Frank Lu}
\begin{document}

\maketitle

\begin{abstract}
In this paper, we prove that the \'{e}tale fundamental group of the N\'{e}ron model of an abelian variety over a number field $K$ is the semidirect product of a finite group with the \'{e}tale fundamental group of the ring of integers of $K.$ We prove this by studying how the Faltings height of an abelian variety changes under covers that spread out to finite \'{e}tale covers of its N\'{e}ron model. We then strengthen this result for elliptic curves. Using Merel's torsion theorem, we show the size of this finite group can be uniformly bounded for a fixed number field. We conclude by giving the list of all possible \'{e}tale fundamental groups for the N\'{e}ron model of an elliptic curve over $\mathbb{Q}.$
\end{abstract}
\tableofcontents
\section{Introduction}
Let $K$ be a number field with ring of integers $\mathcal{O}_K,$ and let $\mathcal{A}/\mathcal{O}_K$ be the N\'{e}ron model of an abelian variety $A/K.$ Consider the kernel of the map $\piet(\mathcal{A}) \rightarrow \piet(\Spec \mathcal{O}_K),$ which we notate by $\pietgeo(\mathcal{A})$ (and discuss in more detail in Section \ref{sec: prelims}). Our goal is to study finiteness properties of this subgroup, with the first main theorem of our paper being the following. 
\begin{thm}\label{thm: basic finiteness}
Let $A$ be an abelian variety over $K$ of dimension $g,$ and let $\mathcal{A}$ be its N\'{e}ron model over $\Spec \mathcal{O}_K.$  Then, $\pietgeo(\mathcal{A})$ is a finite group.
\end{thm}
\par Our main strategy is to study the behavior of the Faltings height under finite \'{e}tale covers of the N\'{e}ron model $\mathcal{A}.$ Namely, suppose that $\phi: A' \rightarrow A$ is an isogeny of abelian varieties over $K,$ such that the corresponding morphism of N\'{e}ron models $\mathcal{A}' \rightarrow \mathcal{A}$ is \'{e}tale. Then, using Faltings's isogeny formula, \cite[Lemma 5]{faltingsmordell}, we obtain $h(A') = h(A) - \frac{1}{2}\log \deg \phi.$ However, if we had such isogenies with arbitrarily large degree, we could produce abelian varieties over $K$ with arbitrarily small height. This is a contradiction of the Northcott property for abelian varieties over $K,$ as there are only finitely many isomorphism classes of abelian varieties with height at most a given constant.
\par We compare this to previous results. In \cite{BCfaaSurfaces}, Bost and Charles show that if $\mathcal{E}$ is an integral regular model over $\mathcal{O}_K$ of an elliptic curve $E/K,$ then $\pietgeo(\mathcal{E})$ is finite. They explicitly state this for $K = \mathbb{Q},$ but they remark that their methods carry over to the general $K$ case. Their approach is to study a condition called pseudoconcavity, a positivity condition, for formal-analytic arithmetic surfaces. Meanwhile, the work of Katz and Lang in \cite{KatzLang} implies that if $A/\mathbb{Q}$ is an abelian variety with N\'{e}ron model $\mathcal{A}/\mathbb{Z},$ then $\piet(\mathcal{A})$ is finite (as observed in \cite[Discussion after Corollary 9.3.4]{BCfaaSurfaces}). This result is obtained by studying the group of Galois coinvariants of the adelic Tate module of an abelian variety over any finitely generated field, as in \cite[Theorem 1 (bis)]{KatzLang}. However, their work doesn't imply finiteness of $\pietgeo(\mathcal{A})$ for arbitrary number fields because the coinvariant group only captures information about covers of $\mathcal{A}$ where the pre-image of the identity section $\Spec \mathcal{O}_K \rightarrow \mathcal{A}$ is a disjoint union of $\Spec \mathcal{O}_K.$
\par We conjecture that the size of $\pietgeo(\mathcal{A})$ should only depend on $g$ and $K,$ as follows.
\begin{conj}\label{conj: abelian variety uniformity}
Given a positive integer $g \geq 1$ and a number field $K,$ there exists a constant $C,$ depending only on $K$ and $g,$ such that the following holds. Let $A/K$ an abelian variety over $K$ with dimension $g.$ Then, if $\mathcal{A}$ is its N\'{e}ron model over $\mathcal{O}_K,$ then $|\pietgeo(\mathcal{A})| \leq C.$ 
\end{conj}
We prove this conjecture for elliptic curves (that is, the $g = 1$ case), as follows.
\begin{restatable}{thm}{uniformity}\label{thm: uniformity for elliptic curves}
Let $K$ be a number field. There exists a constant $C,$ only depending on $K,$ such that the following holds. Suppose $E$ be an elliptic curve over $K,$ and let $\mathcal{E}$ be its N\'{e}ron model. Then, $|\pietgeo(\mathcal{E})| \leq C.$
\end{restatable}
\par In the case when $K = \mathbb{Q},$ we are able to obtain the following explicit description of what the possible \'{e}tale fundamental groups are.
\begin{restatable}{thm}{classificationQ}\label{thm: classification for ECs over Q}
Suppose that $E/\mathbb{Q}$ is an elliptic curve with N\'{e}ron model $\mathcal{E}/\mathbb{Z}.$ Then, $\piet(\mathcal{E}) \simeq \mathbb{Z}/N\mathbb{Z}$ for some $N \in \{1, 2, 3, 5\}.$ Furthermore, all of these cases are achieved, and $\mathbb{Z}/5$ is only possible if $E$ has (at worst) semistable reduction everywhere.
\end{restatable}
\par We obtain Theorem \ref{thm: uniformity for elliptic curves} by leveraging the uniform torsion theorem of Merel. Given an isogeny $E' \rightarrow E$ of elliptic curves over $K$ which spreads out to a finite \'{e}tale isogeny $\phi$ of the N\'{e}ron models, we look at the kernel of $\phi,$ which will be a finite \'{e}tale cover of $\Spec \mathcal{O}_K.$ Using this, we show that the points in the kernel of $E' \rightarrow E$ are defined over the Hilbert class field of $K.$ But then this puts us in the setting of Merel's theorem, as in \cite{merel}.
\par We note here that our result is weaker than Merel's. For instance, suppose that $E'\rightarrow E$ was an isogeny whose kernel was a $p$-group for some prime $p.$ Then, our statement only applies to settings where the kernel spreads out to be finite \'{e}tale over $\Spec \mathcal{O}_K$ (and in particular \'{e}tale over all places above $p$) whereas in the setting of Merel's theorem there is no such condition. 
\par For Theorem \ref{thm: classification for ECs over Q}, we begin by proving the statement in the case when $E$ has semistable reduction everywhere. We start from the list of possible subgroups in Mazur's torsion theorem (as in \cite{mazurtorsion}) and rule out all the subgroups other than the ones in our theorem statement. First, we show that for a finite \'{e}tale cover of N\'{e}ron models $\mathcal{E}' \rightarrow \mathcal{E}$ of degree $n,$ the discriminant of $E'$ is an $n$th power of that of $E.$ We do this by showing that the number of components in the special fibers at each place where $E$ has bad reduction gets scaled by $n.$ Combining this statement with Silverman's formula for the Faltings height (as in \cite[Proposition 1.1]{silvermanheights}), we conclude that such an isogeny must be either trivial or prime order. 
\par The bulk of the work is to rule out $\mathbb{Z}/7,$ which we do by interpreting the existence of such an isogeny to the existence of a rational point on $X(1)$ which is the image (under appropriate morphisms) of an integral point from two schemes: $\Spec \mathbb{Z}[A, B, C]/(A^2 + B^3 - 1728C^7) - \{A = B = C = 0\}$ and $X_1(7).$ We then follow part of the strategy of \cite{pss}, where the study the equation $A^2 + B^3 = C^7.$ Namely, we show that rational points from $X_1(7)$ must be the images of rational points from certain twists of the modular curve $X(7).$ We then verify with Sage code that such points satisfy certain $p$-adic conditions for $p = 2, 3, 7$ which are incompatible with being the image of an integral point from $\Spec \mathbb{Z}[A, B, C]/(A^2 + B^3 - 1728C^7) - \{A = B = C = 0\}.$ 
\par In the case where $E$ has a place of additive reduction, bounds on the component group of the geometric special fiber of $\mathcal{E}$ at that place give us the bound $|\piet(\mathcal{E})| \leq 4.$ To show that $4$ isn't attained as the degree of a finite \'{e}tale cover $\mathcal{E}' \rightarrow \mathcal{E},$ we again employ Silverman's formula, using Tate's algorithm for computing the special fiber (as in \cite[Chapter IV, \S 9]{advAEC}) to relate the discriminants of the source and target.
\par The structure of our paper is as follows. In Section \ref{sec: prelims}, we begin by defining the central object of our paper, $\pietgeo(\bullet),$ and proving some basic facts about it. We then turn to some useful properties of finite \'{e}tale morphisms over N\'{e}ron models of abelian varieties, including the fact that finite \'{e}tale covers of N\'{e}ron models of abelian varieties are still N\'{e}ron models of abelian varieties. In the next section, we prove Theorem \ref{thm: basic finiteness}. 
\par In the last two sections, we study elliptic curves in particular. We begin Section \ref{sec: elliptic curves} by describing how the discriminant of an elliptic curve changes when we pass to covers that spread out to finite \'{e}tale covers of the N\'{e}ron model. We then prove Theorem \ref{thm: uniformity for elliptic curves}. We conclude the paper by proving Theorem \ref{thm: classification for ECs over Q} and giving examples (using data from \cite{lmfdb}) that attain the listed \'{e}tale fundamental groups.
\subsection*{Acknowledgements}
The author would like to thank Mark Kisin for useful conversations and for feedback on earlier drafts of this paper. The author would also like to thank Xinyu Fang for useful discussions about \cite{pss}.
\section{Finite \'{E}tale Morphisms and N\'{e}ron Models}\label{sec: prelims}
In this section, we review through some facts from the theory of \'{e}tale fundamental groups, before proving some lemmas about finite \'{e}tale morphisms of abelian varieties and their N\'{e}ron models which we will need. 
\subsection{\'{E}tale Fundamental Groups}\label{subsec: etale fundamental groups}
We begin by briefly discussing some parts of the theory of \'{e}tale fundamental groups which we will be using. The reader interested in more details can refer to \cite[\href{https://stacks.math.columbia.edu/tag/0BQ6}{Chapter 0BQ6}]{stacks-project}, for instance.
\par First, recall that if we have a variety $X$ over $K$ which is geometrically connected, and $\overline{x}$ is a geometric point of $X,$ we have the short exact sequence (e.g. see \cite[\href{https://stacks.math.columbia.edu/tag/0BTX}{Tag 0BTX}]{stacks-project})
\[0 \rightarrow \piet(X_{\overline{K}}, \overline{x}) \rightarrow \piet(X, \overline{x}) \rightarrow G_K \rightarrow 0,\] where $G_K$ is the absolute Galois group of $K.$ The left group is referred to as the \textbf{geometric \'{e}tale fundamental group} of $X.$ One can think of this group as describing the finite \'{e}tale covers which arise geometrically, rather than from extensions of the base field. 
\par Now, suppose that $X$ is the generic fiber of some integral separated and finite type scheme $\mathcal{X} \rightarrow \Spec \mathcal{O}_K.$ We have a morphism $\piet(\mathcal{X}, \overline{x}) \rightarrow \piet(\Spec\mathcal{O}_K),$ with the latter being the Galois group of $K^{\text{unr}},$ the maximal everywhere unramified extension of $K,$ over $K$ (for instance, this follows from \cite[\href{https://stacks.math.columbia.edu/tag/0BQM}{Tag 0BQM}]{stacks-project}).
\par Combining this morphism with the short exact sequence described above yields the following diagram:
\[\begin{tikzcd}
	1 & {\piet(X_{\overline{K}}, \overline{x})} & {\piet(X, \overline{x})} & {G_K} & 1 \\
	&& {\piet(\mathcal{X}, \overline{x})} & {\piet(\Spec \mathcal{O}_K)}
	\arrow[from=1-1, to=1-2]
	\arrow[from=1-2, to=1-3]
	\arrow[from=1-3, to=1-4]
	\arrow[from=1-3, to=2-3]
	\arrow[from=1-4, to=1-5]
	\arrow[from=1-4, to=2-4]
	\arrow[from=2-3, to=2-4]
\end{tikzcd}\]
Since we know that the rightmost vertical map and the right map on the top row are both surjective, it follows that $\piet(\mathcal{X}, \overline{x}) \rightarrow \piet(\Spec \mathcal{O}_K)$ is also surjective. One would expect from the homotopy exact sequence (see \cite[\href{https://stacks.math.columbia.edu/tag/0C0J}{Tag 0C0J}]{stacks-project}), but we cannot directly use this result as we don't assume we have geometrically connected fibers.
\par Now, we will find it useful to have a notion for the ``geometric" part of the \'{e}tale fundamental group for this scheme $\mathcal{X},$ analogously to the notion of the geometric \'{e}tale fundamental group of a variety over $K.$ Inspired by the short exact sequence that we introduced earlier, we make the following definition (which we mentioned briefly in the introduction).
\begin{defn}
Suppose $\mathcal{X}$ is an integral separated and finite type scheme over $\Spec \mathcal{O}_K,$ which surjects onto $\Spec \mathcal{O}_K$ and whose generic fiber is a variety over $K.$ We let $\pietgeo(\mathcal{X}, \overline{x})$ be the kernel of the map $\piet(\mathcal{X}, \overline{x}) \rightarrow \piet(\Spec \mathcal{O}_K),$ and refer to this as the \textbf{geometric \'{e}tale fundamental group} of $\mathcal{X}.$
\end{defn}
\begin{rmk}
This notion is independent of basepoint in the following manner. Suppose that $\overline{x'}$ was another geometric basepoint of $\mathcal{X}.$ Recall (e.g. from \cite[\href{https://stacks.math.columbia.edu/tag/0BNC}{Tag 0BNC}]{stacks-project}) that the \'{e}tale fundamental group with basepoint at $\overline{x}$ is given by the automorphism group of the fiber functor $F_{\overline{x}},$ which is the functor from the category of finite \'{e}tale covers of $\mathcal{X}$ to the category of sets given by taking the fiber over $\overline{x}.$ Then, given our two basepoints, we have isomorphisms between the \'{e}tale fundamental groups of $\overline{x}$ and $\overline{x}'$ arising from an isomorphism between the fiber functors $F_{\overline{x}}, F_{\overline{x'}};$ see \cite[\href{https://stacks.math.columbia.edu/tag/0BND}{Tag 0BND}]{stacks-project}, for instance.
\par For any choice of isomorphism between the fiber functors, we get a corresponding isomorphism between the corresponding fiber functors for $\Spec \mathcal{O}_K,$ giving us compatible isomorphisms between the \'{e}tale fundamental groups of $\mathcal{X}, \Spec \mathcal{O}_K$ at the basepoints $\overline{x}, \overline{x}'.$ Using these isomorphisms, $\pietgeo(\mathcal{X}, \overline{x})$ is identified with $\pietgeo(\mathcal{X}, \overline{x}').$ 
\par For this reason, we can be permissive with our choice of basepoint, and we sometimes omit it from our notation.
\end{rmk}
If $X$ has a $K$-rational point, we have a splitting of the short exact sequence above, yielding the decomposition \[\piet(X) = \piet(X_{\overline{K}}) \rtimes G_K.\] Similarly, if $\mathcal{X}$ has a $\Spec \mathcal{O}_K$ section (for instance, if $\mathcal{X}$ is the N\'{e}ron model of an abelian varietie), we have the semidirect product expression \[\piet(\mathcal{X}) = \pietgeo(\mathcal{X}) \rtimes \Spec \mathcal{O}_K,\] again from the splitting of a short exact sequence.
\par In the case where $\mathcal{X}$ is normal, the morphism $X \rightarrow \mathcal{X}$ induces a surjection of \'{e}tale fundamental groups. See \cite[\href{https://stacks.math.columbia.edu/tag/0BQM}{Tag 0BQM}]{stacks-project}, for instance; the point is that we can identify these \'{e}tale fundamental groups as being Galois groups of certain extensions of the function field of $\mathcal{X}$ (which is also the function field of $X$). We finish this subsection by showing that, under an additional hypothesis, we may say the same about the geometric \'{e}tale fundamental groups.
\begin{lemma}\label{lem: generic surjects onto integral for geometric}
Let $\mathcal{X} \rightarrow \Spec \mathcal{O}_K$ be a separated, surjective, and finite type morphism from an integral normal scheme. Suppose its generic fiber $X \rightarrow \Spec K$ is geometrically connected, and $\mathcal{X}$ has a $\mathcal{O}_K$-point. Then, we have a map $\piet(X_{\overline{K}}) \rightarrow \pietgeo(\mathcal{X})$ induced from $X_{\overline{K}} \rightarrow \mathcal{X},$ which furthermore is a surjection.
\end{lemma}
\begin{proof}
The fact that this map exists follows because $\piet(X_{\overline{K}}) \rightarrow \piet(X) \rightarrow G_K \rightarrow \piet(\Spec \mathcal{O}_K)$ is zero, and so the map $\piet(X_{\overline{K}}) \rightarrow \piet(\mathcal{X})$ lands in the kernel of $\piet(\mathcal{X}) \rightarrow \piet(\Spec \mathcal{O}_K).$
\par To show surjectivity, suppose that $g \in \pietgeo(\mathcal{X});$ we know that we can lift this to some element $\tilde{g}$ of $\piet(X)$ by the normality assumption. Let $\overline{g}$ be its image in $G_K;$ we know that $\overline{g}$ lies in the kernel of $G_K \rightarrow \piet(\Spec \mathcal{O}_K).$ By assumption, we have a $\Spec \mathcal{O}_K$ point $x: \Spec \mathcal{O}_K \rightarrow \mathcal{X}$ (giving us a rational point $x_K: \Spec K \rightarrow X$) Using the morphism $x_K,$ we can map $\overline{g}$ back to some element $(x_K)_*\overline{g} \in \piet(X).$ Since $\overline{g}$ is annihilated under the map to $\piet(\Spec \mathcal{O}_K),$ using this commutative diagram
\[\begin{tikzcd}
	{G_K} && {\piet(X)} \\
	\\
	{\piet(\Spec \mathcal{O}_K)} && {\piet(\mathcal{X})}
	\arrow["{(x_K)_*}", from=1-1, to=1-3]
	\arrow[from=1-1, to=3-1]
	\arrow[from=1-3, to=3-3]
	\arrow["{x_*}"', from=3-1, to=3-3]
\end{tikzcd}\]
we know that $(x_K)_*\overline{g}$ lies in the kernel of the map $\piet(X) \rightarrow \piet(\mathcal{X}).$ The element $\tilde{g} ((x_K)_*\overline{g})^{-1}$ is therefore an element of $\piet(X_{\overline{K}})$ which maps to $g,$ proving surjectivity.
\end{proof}
\begin{rmk}
Strictly speaking, the above argument handles the \'{e}tale fundamental groups which have basepoint of $X_{\overline{K}}$ given on a geometric point above $x.$ However, once again, note that if we pick a different basepoint on $X_{\overline{K}},$ we can identify \'{e}tale fundamental groups at different basepoints via isomorphisms of fiber functors, which can be chosen compatibly across the different schemes (again utilizing \cite[\href{https://stacks.math.columbia.edu/tag/0BND}{Tag 0BND}]{stacks-project}). We thus obtain the same conclusion of surjectivity of $\piet(X_{\overline{K}}) \rightarrow \pietgeo(\mathcal{X})$ for this different choice of basepoint. 
\end{rmk} 
\subsection{Finite \'{E}tale Morphisms of N\'{e}ron Models}
We now restrict our focus to N\'{e}ron models of abelian varieties. First, recall the definition of the N\'{e}ron model: for an abelian variety $A/K,$ this is a smooth separated and finite type group scheme $\mathcal{A}$ over $\Spec \mathcal{O}_K$ satisfying the following universal property. If $\mathcal{Y} \rightarrow \Spec \mathcal{O}_K$ is smooth with generic fiber $Y,$ then a morphism $Y \rightarrow A$ over $K$ spreads out uniquely to a morphism $\mathcal{Y} \rightarrow \mathcal{A}$ (e.g. \cite[\S 1.2, Definition 1]{BLRNeron}). For a more thorough introduction on N\'{e}ron models, the reader can consult \cite{BLRNeron} for more information. 
\par We now collect some results about morphisms of abelian varieties and their N\'{e}ron models. We start with the following lemma, which is probably known to experts, but which the author couldn't find a reference for.
\begin{lemma}\label{lem: generic fiber finite etale covers are abelian}
Suppose that $A$ is an abelian variety over a field $K$ of characteristic zero, and $A' \rightarrow A$ is a connected finite \'{e}tale cover such that $A'/K$ has a $K$-point. Then, $A'$ can be given the structure of an abelian variety, such that $A' \rightarrow A$ is an isogeny. \end{lemma}
\begin{proof}
Let $f: A' \rightarrow A$ be our finite \'{e}tale cover. Start by base changing this to $\overline{K},$ yielding a finite \'{e}tale cover $f_{\overline{K}}: A'_{\overline{K}} \rightarrow A_{\overline{K}};$ by \cite[\href{https://stacks.math.columbia.edu/tag/04KV}{Tag 04KV}]{stacks-project}, this is geometrically connected. The fact that this morphism is an isogeny of abelian varieties is a result of Serre and Lang (e.g. see \cite[\S 18, Theorem]{mumfordAV}). 
\par We now show that the claim holds over $K;$ the idea will be to show that all the data can be descended down to $K.$ By the above result, we know that we may equip $A'_{\overline{K}}$ with the structure of an abelian variety, such that $A'_{\overline{K}} \rightarrow A_{\overline{K}}$ is an isogeny of abelian varieties. This structure is given by the data of a multiplication map $m': A'_{\overline{K}} \times_{\overline{K}} A'_{\overline{K}} \rightarrow A'_{\overline{K}},$ an inverse map $i': A'_{\overline{K}} \rightarrow A'_{\overline{K}},$ and an identity section $e': \overline{K} \rightarrow A'_{\overline{K}},$ such that the cover $A' \rightarrow A$ is compatible with these morphisms when we base change the cover to $\overline{K}.$ We first show that there exists a choice of these three pieces of data which descends to $K,$ and thus makes $A'$ an abelian variety over $K.$
\par We start by showing that $A'_{\overline{K}}$ has a group structure such that the identity section $\Spec \overline{K} \rightarrow A'_{\overline{K}}$ is the base change of a map $\Spec K \rightarrow A'.$ By assumption, $A'$ has a $K$-rational point $p: \Spec K \rightarrow A',$ and its image under $A' \rightarrow A$ is also $K$-rational point. Translating by this point on $A$ gives us an isomorphism which we can base change to get an isomorphism of $A'.$ This isomorphism then takes some point $q: \Spec K \rightarrow A',$ which maps to the identity point in $A,$ and sends it to $p.$ Thus, the fiber of $A' \rightarrow A$ over the identity section of $A$ has a $K$-rational point. We may translate our group structure on $A'_{\overline{K}}$ such that the base change of $q$ to an $\overline{K}$-point is the identity. Note that the morphism $A'_{\overline{K}} \rightarrow A_{\overline{K}}$ will still be an isogeny with this new group structure on $A'_{\overline{K}},$ since it is the translation of an isogeny that sends the origin to itself.
\par We now claim this new group structure can be descended to $K.$ To do this, we apply Galois descent. Given an automorphism $\sigma$ of $\overline{K}$ fixing $K,$ let $m'_{\sigma}, i'_{\sigma}$ be the base change of the morphisms $m': A'_{\overline{K}} \times A'_{\overline{K}} \rightarrow A'_{\overline{K}}$ and $i': A'_{\overline{K}} \rightarrow A'_{\overline{K}}$ along the isomorphism $\sigma: \overline{K} \rightarrow \overline{K},$ respectively. Since $f_{\overline{K}}$ is an isogeny, if $m: A \times A \rightarrow A$ is the multiplication morphism for $A,$ we have the following commutative diagrams, the right one being the base change of the left one along $\sigma:$
\[\begin{tikzcd}
	{A'_{\overline{K}} \times_{\overline{K}} A'_{\overline{K}}} && {A'_{\overline{K}}} && {A'_{\overline{K}} \times_{\overline{K}} A'_{\overline{K}}} && {A'_{\overline{K}}} \\
	\\
	{A_{\overline{K}} \times_{\overline{K}} A_{\overline{K}}} && {A_{\overline{K}}} && {A_{\overline{K}} \times_{\overline{K}} A_{\overline{K}}} && {A_{\overline{K}}}
	\arrow["{m'}", from=1-1, to=1-3]
	\arrow["{f_{\overline{K}} \times f_{\overline{K}}}"{description}, from=1-1, to=3-1]
	\arrow["{f_{\overline{K}}}"{description}, from=1-3, to=3-3]
	\arrow["{m'_{\sigma}}", from=1-5, to=1-7]
	\arrow["{f_{\overline{K}} \times f_{\overline{K}}}"{description}, from=1-5, to=3-5]
	\arrow["{f_{\overline{K}}}"{description}, from=1-7, to=3-7]
	\arrow["{m_{\overline{K}}}"', from=3-1, to=3-3]
	\arrow["{m_{\overline{K}}}"', from=3-5, to=3-7]
\end{tikzcd}\]
Note that the only morphism that isn't already defined over $K$ is the top morphism, which is why the other three parts of the diagram are the same.
\par We can now combine these into the following diagram:
\[\begin{tikzcd}
	{A'_{\overline{K}} \times_{\overline{K}} A'_{\overline{K}}} &&&& {A'_{\overline{K}}} \\
	\\
	{A_{\overline{K}} \times_{\overline{K}} A_{\overline{K}}} &&&& {A_{\overline{K}}}
	\arrow["{m' \circ (m'_{\sigma}, i' \circ m')}", from=1-1, to=1-5]
	\arrow["{f_{\overline{K}} \times f_{\overline{K}}}"{description}, from=1-1, to=3-1]
	\arrow["{f_{\overline{K}}}"{description}, from=1-5, to=3-5]
	\arrow["{m_{\overline{K}} \circ (m_{\overline{K}}, i \circ m_{\overline{K}})}"', from=3-1, to=3-5]
\end{tikzcd}\]
Note that the bottom map is equal to $e_{\overline{K}};$ therefore, it follows that the top map factors through the base change of $e_{\overline{K}}: \Spec \overline{K} \rightarrow A_{\overline{K}}$ along $f_{\overline{K}},$ which is a finite disjoint union of $\Spec \overline{K}$ (since $f_{\overline{K}}$ is a finite \'{e}tale morphism). However, $A'_{\overline{K}}$ is connected, and thus $A'_{\overline{K}} \times_{\overline{K}} A'_{\overline{K}}$ is too (e.g. using \cite[Lemma 10.5.9]{vakil}). 
\par Hence, the image of the top map must be one of the $\Spec \overline{K};$ since our choice of unit section $e': \Spec \overline{K} \rightarrow A'_{\overline{K}}$ is defined over $K,$ we see that the image of $(e', e')$ is $e'$ (as $e'$ will be the identity section under the group law $m'_{\sigma}$ as well). Therefore, $m' \circ (m'_{\sigma}, i' \circ m')$ factors through $e',$ which means that $m'_{\sigma} = m'.$ But this holds for all $\sigma,$ which means by Galois descent that $m'$ descends to a morphism over $K.$ The exact same argument for $i'$ with this diagram also shows that $i'$ descends to $K.$
\[\begin{tikzcd}
	{A'_{\overline{K}}} && {A'_{\overline{K}}} \\
	\\
	{A_{\overline{K}}} && {A_{\overline{K}}}
	\arrow["{m' \circ (i'_{\sigma}, i' \circ i')}", from=1-1, to=1-3]
	\arrow["{f_{\overline{K}}}"{description}, from=1-1, to=3-1]
	\arrow["{f_{\overline{K}}}"{description}, from=1-3, to=3-3]
	\arrow["{m_{\overline{K}} \circ (i_{\overline{K}}, i_{\overline{K}} \circ i_{\overline{K}})}"', from=3-1, to=3-3]
\end{tikzcd}\]
Thus, $A'$ can be given the structure of an abelian variety (as one can then descend the commutative diagrams giving the various axioms these morphisms satisfy) such that $f: A' \rightarrow A$ becomes an isogeny, as it is compatible with the group structures over $K$ (since it is compatible over $\overline{K}$).
\end{proof}
While isogenies of abelian varieties of number fields are always finite \'{e}tale (e.g. by \cite[\S 7.3, Lemmas 1, 2]{BLRNeron}), this will not remain true when we spread out to the N\'{e}ron model. Indeed, the following (probably well-known) lemma below tells us that a lot of these morphisms do not spread out to \'{e}tale morphisms.
\begin{lemma}\label{lem: multiplication by N not etale}
Suppose that $\phi: A' \rightarrow A$ is an isogeny of abelian varieties over a number field $K$ whose kernel contains $A'[n]$ for some integer $n \geq 2.$ Then, the corresponding morphism of N\'{e}ron models $\mathcal{A}' \rightarrow \mathcal{A}$ over $\mathcal{O}_K$ is not \'{e}tale.
\end{lemma}
\begin{proof}
First, we claim that for $n \geq 2,$ the morphism $[n]$ is not \'{e}tale. To check this, let $\mathfrak{p}$ be a prime of $\mathcal{O}_K$ dividing $n,$ and consider the special fiber of $\mathcal{A}' \xrightarrow{[n]} \mathcal{A}',$ which is multiplication by $n$ on the special fiber. By \cite[\href{https://stacks.math.columbia.edu/tag/0BF5}{Tag 0BF5}]{stacks-project}, we know that $[n]$ induces multiplication by $n$ on tangent spaces at the identity point. But $\mathfrak{p}|n,$ so this map is zero, and in particular not injective. Thus, by \cite[\href{https://stacks.math.columbia.edu/tag/0B2G}{Tag 0B2G}]{stacks-project}, $[n]$ is not unramified and thus not \'{e}tale over the fiber above $\mathfrak{p}.$ So $[n]: \mathcal{A}' \rightarrow \mathcal{A}'$ cannot be \'{e}tale.
\par To get the general statement, if the kernel of $\phi$ contains $[n],$ then we can factor $\phi$ as \[A' \xrightarrow{[n]} A' \rightarrow A\] and then extend these morphisms to N\'{e}ron models. Then, apply the cotangent exact sequence to see that \[\Omega^1_{\mathcal{A}'/\mathcal{A}} \rightarrow \Omega^1_{\mathcal{A}' \xrightarrow{[n]} \mathcal{A}'} \rightarrow 0\] is exact. In particular, as the second term is nonzero, the first term is too, so $\mathcal{A}' \rightarrow \mathcal{A}$ is not \'{e}tale, as desired.
\end{proof}
In the other direction, suppose we are actually given a connected finite \'{e}tale cover $\mathcal{A}' \rightarrow \mathcal{A}$ of the N\'{e}ron model $\mathcal{A}$ of an abelian variety $A.$ If we know that the cover has a $\Spec \mathcal{O}_K$ section, then by Lemma \ref{lem: generic fiber finite etale covers are abelian}, the generic fiber of this morphism is an isogeny of abelian varieties. In this case, it turns out that $\mathcal{A}'$ is also the N\'{e}ron model of its generic fiber, which we will prove shortly.
\begin{prop}\label{prop: finite etale covers are neron models}
Let $\mathcal{A}$ be the N\'{e}ron model of an abelian variety $A$ over a number field $K.$ Let $\mathcal{A}' \rightarrow \mathcal{A}$ be a connected finite \'{e}tale cover such that $\mathcal{A}'$ has a $\Spec \mathcal{O}_K$-point. Then, $\mathcal{A}'$ is the N\'{e}ron model of its generic fiber.
\end{prop}
\begin{proof}
By Lemma \ref{lem: generic fiber finite etale covers are abelian}, we know that the generic fiber of $f: \mathcal{A}' \rightarrow \mathcal{A}$ is an isogeny of abelian varieties $A' \rightarrow A.$ Let $\mathcal{N}$ be the N\'{e}ron model of $A';$ then, by the N\'{e}ron mapping property on both $\mathcal{A}, \mathcal{N},$ since $\mathcal{A}'$ is smooth separated and finite type (being finite \'{e}tale over the smooth separated and finite type $\mathcal{A}$), we have the following commutative diagram:
\[\begin{tikzcd}
	{\mathcal{A}'} && {\mathcal{N}} \\
	\\
	& {\mathcal{A}}
	\arrow["g"{description}, from=1-1, to=1-3]
	\arrow["f"{description}, from=1-1, to=3-2]
	\arrow["{f'}"{description}, from=1-3, to=3-2]
\end{tikzcd}\]
However, notice that $g$ is finite, since $f$ is finite and $f'$ is separated (which itself follows since $f'$ is a morphism between separated schemes over $\mathcal{O}_K$). Furthermore, $g$ extends the identity morphism over the generic fiber, and so in particular is birational. We also note that $\mathcal{A}'$ and $\mathcal{N}$ are integral; reducedness follows from \cite[\href{https://stacks.math.columbia.edu/tag/034E}{Tag 034E}]{stacks-project}, and irreducibility follows from the fact that they are smooth over the connected $\Spec \mathcal{O}_K,$ so their generic points must lie in their generic fibers (e.g. using \cite[\href{https://stacks.math.columbia.edu/tag/03HV}{Tag 03HV}]{stacks-project}), meaning that the only generic point for either scheme is the generic point of $A'.$ Furthermore, $\mathcal{A}', \mathcal{N}$ are normal (since they are smooth over $\Spec \mathcal{O}_K,$ so we can use \cite[\href{https://stacks.math.columbia.edu/tag/034F}{Tag 034F}]{stacks-project}, for instance). Therefore, by \cite[\href{https://stacks.math.columbia.edu/tag/0AB1}{Tag 0AB1}]{stacks-project}, $g$ is actually an isomorphism, so thus $\mathcal{A}'$ is the N\'{e}ron model of its generic fiber, as desired.
\end{proof}
We conclude with a factorization lemma for finite \'{e}tale morphisms between N\'{e}ron models.
\begin{lemma}\label{lem: finite etale factors into finite etales}
Suppose that $A' \rightarrow A'' \rightarrow A$ is a sequence of isogenies of abelian varieties over $K,$ with the corresponding sequence of morphisms of their N\'{e}ron models $\mathcal{A}' \rightarrow \mathcal{A}'' \rightarrow \mathcal{A}.$ Suppose furthermore the composition $\mathcal{A}' \rightarrow \mathcal{A}$ is finite \'{e}tale. Then, each of the two morphisms in the composition are also finite \'{e}tale.
\end{lemma}
\begin{proof}
First, to show that $\mathcal{A}'' \rightarrow \mathcal{A}$ is finite \'{e}tale, recall that $\piet(A) \rightarrow \piet(\mathcal{A})$ is surjective (since $\mathcal{A}$ is integral), meaning that by \cite[\href{https://stacks.math.columbia.edu/tag/0BN6}{Tag 0BN6}]{stacks-project} the category of finite \'{e}tale covers of $\mathcal{A}$ embeds fully faithfully into the category of finite \'{e}tale covers of $A$ via restriction to the generic fiber. In particular, it suffices to show that, viewing the finite \'{e}tale covers as $\piet$ sets, that the $\piet(A)$-set corresponding to $A'' \rightarrow A$ arises as a $\piet(\mathcal{A})$ set, and that $\mathcal{A}'' \rightarrow \mathcal{A}$ corresponds to this $\piet(\mathcal{A})$ set.
\par To check this, note that the morphisms $A' \rightarrow A'' \rightarrow A$ correspond to a morphism of $\piet(A)$ sets $\piet(A)/N' \rightarrow \piet(A)/N''.$ We also know that the former can be viewed as a $\piet(\mathcal{A})$-set, meaning that the kernel of $\piet(A) \rightarrow \piet(\mathcal{A})$ acts trivially on it. But this means that this kernel has to act trivially on $\piet(A)/N'',$ since this map is $\piet(A)$-equivariant between sets with transitive $\piet(A)$ action. Hence, we see that $A'' \rightarrow A$ arises as the generic fiber of some finite \'{e}tale cover of $\mathcal{A}.$ However, by the full faithfulness and by Proposition \ref{prop: finite etale covers are neron models}, $\mathcal{A}'' \rightarrow \mathcal{A}$ is actually this finite \'{e}tale cover.
\par First, note that since $\mathcal{A}' \rightarrow \mathcal{A}$ is finite, and all these schemes are separated, we have that $\mathcal{A}' \rightarrow \mathcal{A}''$ is finite (see \cite[Theorem 11.1.1]{vakil}, for instance). The fact that $\mathcal{A}' \rightarrow \mathcal{A}''$ is \'{e}tale follows from \cite[\href{https://stacks.math.columbia.edu/tag/02GW}{Tag 02GW}]{stacks-project}, since we have already shown that $\mathcal{A}'' \rightarrow \mathcal{A}$ is \'{e}tale.
\end{proof}
\section{Proof of Theorem \ref{thm: basic finiteness}}
In this section, we prove Theorem \ref{thm: basic finiteness}. We prove this by showing that, when we take finite \'{e}tale covers of the N\'{e}ron model of an abelian variety, then the Faltings height of the generic fibers, as introduced in \cite{faltingsmordell}, decrease as we go up these covers.
\par To begin, recall that $\piet(A_{\overline{K}})$ is given by $\hat{\mathbb{Z}}^{2g} \simeq \prod\limits_{l \text{ prime}} \mathbb{Z}_l^{2g}$ (see \cite[p.171]{mumfordAV}, for instance). This group is a topologically finitely generated and abelian, so by Lemma \ref{lem: generic surjects onto integral for geometric}, $\pietgeo(\mathcal{A})$ is as well, since it is a quotient of $\piet(A_{\overline{K}}).$ We therefore have finite index subgroups $N \cdot \pietgeo(\mathcal{A})$ of $\pietgeo(\mathcal{A})$ and $N \cdot \pietgeo(\mathcal{A}) \rtimes \piet(\Spec \mathcal{O}_K)$ of $\piet(\mathcal{A})$ for each integer $N \geq 2.$ 
\par The subgroup $N \cdot \pietgeo(\mathcal{A}) \rtimes \piet(\Spec \mathcal{O}_K) \subset \piet(\mathcal{A})$ corresponds to a connected finite \'{e}tale cover $\phi_N: \mathcal{A}_N \rightarrow \mathcal{A}.$ Since the finite index subgroup contains the image of $\piet(\Spec \mathcal{O}_K)$ under the section $\Spec \mathcal{O}_K \rightarrow \mathcal{A}$ by construction, the base change of this cover along the section becomes a trivial cover. In particular, this implies that $\mathcal{A}_N$ has a $\Spec \mathcal{O}_K$-point. Lemma \ref{lem: generic fiber finite etale covers are abelian} and Proposition \ref{prop: finite etale covers are neron models} then tell us that the $\mathcal{A}_N$ are N\'{e}ron models of the corresponding morphism of the abelian varieties on the generic fiber. 
\par Let $A_N$ be the generic fiber of $\mathcal{A}_N,$ and let $f_N$ be the generic fiber of the morphism $\phi_N.$ Note that  \[\deg f_N = \deg \phi_N = |\pietgeo(\mathcal{A})/N\pietgeo(\mathcal{A})|.\] To show that $\pietgeo(\mathcal{A})$ is finite, it suffices to show that the $\deg f_N$ are bounded from above. Otherwise, since $\pietgeo(\mathcal{A})$ is topologically finitely generated and abelian, taking larger $N$ will allow us to produce arbitrarily large quotients $\pietgeo(\mathcal{A})/N\pietgeo(\mathcal{A}).$ 
\par Furthermore, once we have finiteness, if we take $N = |\pietgeo(\mathcal{A})|,$ then the degree of $f_N$ is precisely the size of $\pietgeo(\mathcal{A}).$ In particular, bounds on the degrees of the $f_N$ are also bounds on the size of $\pietgeo(\mathcal{A}).$
\par We now show the degree of $f_N,$ over all $N,$ is bounded from above. Since $\phi_N$ is finite \'{e}tale, we know that by Faltings's isogeny formula, \cite[Lemma 5]{faltingsmordell}, that \[h(A) = h(A_N) + \frac{1}{2} \log \deg f_N.\] Note that this is only stated in the case when $A$ has semistable reduction everywhere, but the same argument given in \cite{faltingsmordell} also holds for general $A.$ See also \cite[Proposition 1.4.1]{raynaud}.
\par Let $L/K$ be a finite degree extension such that $A$ attains semistable reduction everywhere. Then, we know that $A_N$ also has semistable reduction everywhere over $L$ as well, since $A_N$ is isogeneous to $A$ (and by using \cite[\S 7.3, Corollary 7]{BLRNeron}). Now, if the degree of $f_N$ is allowed to be arbitrarily large, this means that we have abelian varieties $A_N$ over $K$ with arbitrarily small heights. Furthermore, for all positive integers $N,$ we have that $h((A_N)_L) \leq h(A_N)$ (see \cite[Remark 5.1.1]{ChaiModuli}). In particular, we have produced abelian varieties $(A_N)_L$ over the number field $L$ that have semistable reduction everywhere and with arbitrarily small heights.
\par However, this is not possible because of the Northcott property for abelian varieties, which is \cite[Theorem 1]{faltingsmordell} combined with Zarhin's trick to equip $A^4 \times (A^{\vee})^4$ (where $A^{\vee}$ is the dual of $A$) with a principal polarization. In particular, there is a minimum height for abelian varieties over $L.$ Therefore, $h((A_N)_L),$ and thus $h(A_N),$ cannot get arbitrarily small, so the degree of $A_N \rightarrow A$ is bounded. This shows that $\pietgeo(\mathcal{A})$ is finite, and thus proves Theorem \ref{thm: basic finiteness}.
\begin{rmk}
Using this argument, one can show that there is a bound which only depends on $h(A)$ and $g,$ if $A$ is assumed to have semistable reduction everywhere. For instance, the lower bound of Bost, given in \cite[Corollary 8.4]{GRperiods}, gives us $\frac{1}{2} \log \deg f_N = h(A) - h(A_N) \leq h(A) + \log (\pi \sqrt{2}) g.$ 
\par We can also use the work of \cite{raynaud} to show that this only depends on the places of bad reduction of $A,$ as well as $g$ and $K.$ This follows from \cite[Theorem 4.4.9]{raynaud}, which gives the maximum difference between the height of two abelian varieties isogenous over $K$ in terms of the places of bad reduction for $A,$ as well as $g$ and the places where $K/\mathbb{Q}$ is ramified.
\end{rmk}
\section{N\'{e}ron Models of Elliptic Curves}\label{sec: elliptic curves}
For the rest of the paper, we restrict our attention to elliptic curves. Additional results for torsion points of elliptic curves, and more explicit descriptions of the N\'{e}ron model of an elliptic curve, allow us to say more about the \'{e}tale fundamental groups of N\'{e}ron models of elliptic curves. We will start with a criterion to determine whether a morphism $E' \rightarrow E$ spreads out to a finite \'{e}tale morphism of their N\'{e}ron models. We will then show that the finiteness result proved in the previous section can be upgraded to a uniformity statement in the case of elliptic curves.
\subsection{A Numerical Condition for Finite \'{E}tale Morphisms}
\par To begin, we show how the discriminant of an elliptic curve changes when we take a finite \'{e}tale cover of N\'{e}ron models, in the case of semistable reduction.
\begin{lemma}\label{lem: component grows under finite etale}
Suppose that $f: \mathcal{E}' \rightarrow \mathcal{E}$ is a connected finite \'{e}tale cover of the N\'{e}ron model $\mathcal{E}$ of $E$ which has a $\Spec \mathcal{O}_K$ point, and let $E'$ be the generic fiber of $\mathcal{E}'.$ Suppose furthermore that $E$ has semistable or good reduction at a prime $\mathfrak{p}$ of $K.$ Let $\Delta_{E/K}, \Delta_{E'/K}$ be the minimal discriminant of $E$ and $E'$ over $K,$ respectively, and let, for an ideal $I$ of $\mathcal{O}_K,$ $v(I)$ be the exponent of $\mathfrak{p}$ in the factorization of $I$. Then, $v(\Delta_{E'/K}) = (\deg f) v(\Delta_{E/K}).$ 
\end{lemma}
\begin{proof}
First, observe that isogenies preserve the places of good and bad reduction (e.g. see \cite[Chapter VII, Corollary 7.2]{AEC}), as well as semistable reduction (e.g. \cite[\S 7.3]{BLRNeron}). Hence, if $E$ has good reduction, then $(\deg f) v(\Delta_{E/K}) = 0 = v(\Delta_{E'/K}).$ So suppose that $E$ has semistable bad reduction at $\mathfrak{p};$ we want to show that $v(\Delta_{E'/K}) = (\deg f) v(\Delta_{E/K}).$ Let $k_v$ be the residue field of $\mathfrak{p}$ and $\overline{k_v}$ an algebraic closure of $k_v.$
\par We know that the N\'{e}ron model $\mathcal{E}$ of $E$ is equal to the smooth locus of the minimal regular model $\mathcal{C}$ of $E$ over $\mathcal{O}_K;$ see \cite[Chapter IV, Theorem 6.1]{advAEC}. Now, using the Kodaira-N\'{e}ron classification (see \cite[Chapter IV, Theorem 8.2]{advAEC}) and Tate's algorithm (see \cite[Chapter IV, \S 9]{advAEC}), the geometric special fiber $\mathcal{C}_{\overline{k_v}}$ of $\mathcal{C}$ consists of $v(\Delta_{E/K})$ intersecting copies of $\mathbb{P}^1$ of multiplicity one, and (the reduced closed subscheme for) $\mathcal{C} - \mathcal{E}$ is codimension $2,$ consisting of these intersection points.
\par By purity of the branch locus results (specifically \cite[\href{https://stacks.math.columbia.edu/tag/0BMA}{Tag 0BMA}]{stacks-project} and  \cite[\href{https://stacks.math.columbia.edu/tag/0EY7}{Tag 0EY7}]{stacks-project}), the finite \'{e}tale cover $\mathcal{E}' \rightarrow \mathcal{E}$ is the base change of some finite \'{e}tale cover $\mathcal{C}' \rightarrow \mathcal{C}.$ Thus, the geometric special fiber of the morphism $\mathcal{C}' \rightarrow \mathcal{C}$ is a $\deg f$ finite \'{e}tale cover of an arrangement of $v(\Delta_{E/K})$ $\mathbb{P}^1$s. But $\mathbb{P}^1_{\overline{k_v}}$ has trivial $\piet$ (for instance, using the surjectivity of specialization as in \cite[\href{https://stacks.math.columbia.edu/tag/0C0P}{Tag 0C0P}]{stacks-project} on $\mathbb{P}^1_{\mathbb{Z}},$ and the fact that $\piet(\mathbb{P}^1_{\mathbb{C}}) = 1$). 
\par By passing to the normalization of $\mathcal{C}_{\overline{k_v}},$ we find that the geometric special fiber of $\mathcal{C}'$ consists of $(\deg f) v(\Delta_{E/K})$ copies of $\mathbb{P}^1$ which intersect in some arrangement. Furthermore, note that $\mathcal{E}' = \mathcal{E} \times_{\mathcal{C}} \mathcal{C}',$ since both sides are connected finite \'{e}tale covers which are birational to each other. Therefore, $\mathcal{C}' - \mathcal{E}'$ is the pre-image of $\mathcal{C} - \mathcal{E},$ which is a finite set of closed points, and so the special fiber of $\mathcal{E}'$ has $(\deg f) v(\Delta_{E/K})$ geometric components.
\par However, we know that $\mathcal{E}'$ is the N\'{e}ron model of $E'$ by Proposition \ref{prop: finite etale covers are neron models}, and semistability is preserved under isogeny. Running the same argument as we did for $\mathcal{E},$ we know that $\mathcal{E}'$'s special fiber at $v$ has $v(\Delta_{E'/K})$ geometric components. Therefore, we find that $v(\Delta_{E'/K}) = (\deg f) v(\Delta_{E/K}),$ which is what we wanted to show. 
\end{proof}
We next note that, in the additive reduction case, the morphism on the special fiber essentially takes one of two forms.
\begin{lemma}\label{lem: additive reduction iso or crushes to point}
Suppose that $E' \rightarrow E$ is an isogeny of elliptic curves over $K,$ whose corresponding morphism of N\'{e}ron models is $\mathcal{E}' \rightarrow \mathcal{E}$ over $\mathcal{O}_K.$ Let $v$ be a place of additive reduction for $E.$ Then, base changing to $\overline{k_v},$ the algebraic closure of the residue field at $v,$ either each connected component of $\mathcal{E}'_{\overline{k_v}}$ is mapped isomorphically onto a connected component of $\mathcal{E},$ or each one is crushed to a point.
\end{lemma}
\begin{proof}
First, we know that $\mathcal{E}' \rightarrow \mathcal{E}$ is a morphism of group schemes, since its generic fiber is an isogeny and by applying the N\'{e}ron mapping property. Therefore, $\mathcal{E}'_{\overline{k_v}} \rightarrow \mathcal{E}_{\overline{k_v}}$ is a group homomorphism from some number of disjoint copies of $\mathbb{G}_a$ to some number of disjoint copies of $\mathbb{G}_a$. 
\par Now, restricting to the identity components, we obtain a group homomorphism from $\mathbb{G}_a$ to itself. But $\mathbb{G}_a = \Spec \overline{k_v}[x]$ as a scheme, so this morphism is given by sending $x$ to some $f \in \overline{\mathbb{F}_v}[x].$ Since this morphism is a group homomorphism, we find that $f(x + y) = f(x) + f(y)$ in $\overline{\mathbb{F}_v}[x, y],$ so thus $f(x) = cx$ for some $c \in \overline{\mathbb{F}_v}.$ But then either $c = 0,$ so this crushes the component to a point, or $c \neq 0,$ and this is an isomorphism.
\par Since $\mathcal{E}'_{\overline{\mathbb{F}_v}} \rightarrow \mathcal{E}_{\overline{\mathbb{F}_v}}$ is a group homomorphism, appropriate translations by points on the other connected components tell us that this morphism exhibits the same behavior when restricted to any component of the source (and its image in the target). In particular, either each component is isomorphically mapped to another component, or each component is crushed to a point, as desired.
\end{proof}
We think of Lemma \ref{lem: component grows under finite etale} as giving us a necessary condition for the morphism of N\'{e}ron models being finite \'{e}tale. In fact, we also are able to show a sufficient condition as well: the induced morphism of N\'{e}ron models is finite \'{e}tale is some conditions are satisfied, depending on the number of components at the geometric special fibers over each prime of $\mathcal{O}_K$ (which, in the semistable reduction case, is determined by the discriminant by the above lemma) and the Faltings height.
\begin{prop}\label{prop: sufficiency for finite etale}
Suppose that $E' \rightarrow E$ is an isogeny of elliptic curves over $K$ of degree $n.$ Then, if $h(E') = h(E) - \frac{1}{2}\log n,$ and for every place of bad reduction $E'$ has $n$ times as many geometric connected components as $E$ does, then this isogeny spreads out to a finite \'{e}tale cover of N\'{e}ron models. 
\end{prop}
\begin{proof}
Let $\mathcal{E}', \mathcal{E}$ be the N\'{e}ron models of $E', E,$ respectively, and let $f: \mathcal{E}' \rightarrow \mathcal{E}$ denote the morphism on N\'{e}ron models. From the condition on the Faltings height, and using Faltings's isogeny formula (\cite[Lemma 5]{faltingsmordell}), we know that $e^*\Omega^1_{\mathcal{E}'/\mathcal{E}} = 0,$ where $e: \Spec\mathcal{O}_K \rightarrow \mathcal{E}'$ is the identity section. 
\par We now show that $\Omega^1_{\mathcal{E}'/\mathcal{E}} = 0.$ If not, then some stalk is nonzero, and thus for some field point $s:\Spec k \rightarrow \mathcal{E}',$ $s^*\Omega^1_{\mathcal{E}'/\mathcal{E}} \neq 0.$ But this field point induces an isomorphism $t_s: \mathcal{E}'_k \rightarrow \mathcal{E}'_k$ by translation by $s,$ which is the base change of the isomorphism given by translating $\mathcal{E}_k$ by $f(s).$ Thus, if $h: \mathcal{E}'_k \rightarrow \mathcal{E}'$ is the morphism induced by base change along $b: \Spec k \rightarrow \Spec \mathcal{O}_K,$ we have an isomorphism $t_s^*h^*\Omega^1_{\mathcal{E}'/\mathcal{E}} \simeq h^*\Omega^1_{\mathcal{E}'/\mathcal{E}} \simeq \Omega^1_{\mathcal{E}'_k/\mathcal{E}_k}.$ But now pulling back along the identity section of $\mathcal{E}_k,$ we find that $0 \neq s^*\Omega^1_{\mathcal{E}'/\mathcal{E}} \simeq e_k^*\Omega^1_{\mathcal{E}'/\mathcal{E}} = b^*e^*\Omega^1_{\mathcal{E}'/\mathcal{E}} =0,$ contradiction. 
\par Next, we show that the morphism $\mathcal{E}' \rightarrow \mathcal{E}$ is flat; as we know N\'{e}ron models are finite type over $\mathcal{O}_K,$ this is enough to show that this morphism is \'{e}tale. The fibral criterion for flatness (see \cite[\href{https://stacks.math.columbia.edu/tag/039B}{Tag 039B}]{stacks-project}, for instance) tells us it is enough to check flatness on each fiber. If $v$ is a place of multiplicative or good reduction, then this follows from \cite[\S 7.3, Proposition 6]{BLRNeron}, which tells us that $f$ is an isogeny (in the sense of being an isogeny fiber-wise), and thus is flat fiberwise by \cite[\S 7.3, Lemma 1]{BLRNeron}.
\par Otherwise, suppose that $v$ is a place of additive reduction; then, since base changing from the residue field $k_v$ at $v$ to its algebraic closure $\overline{k_v}$ is faithfully flat, it suffices to prove that $\mathcal{E}'_{\overline{k_v}} \rightarrow \mathcal{E}_{\overline{k_v}}$ is flat (e.g. see \cite[\href{https://stacks.math.columbia.edu/tag/02L2}{Tag 02L2}]{stacks-project}). By Lemma \ref{lem: additive reduction iso or crushes to point}, we know that this morphism either isomorphically maps each component of $\mathcal{E}'_{\overline{k_v}}$ onto a component of $\mathcal{E}_{\overline{k_v}},$ or it crushes it to a point. 
\par But we would find in the latter case that the cotangent sheaf $\Omega^1_{\mathcal{E}'_{\overline{k_v}}/ \mathcal{E}_{\overline{k_v}}},$ restricted to one of these components $\mathbb{G}_a \simeq \Spec \overline{k_v}[x],$ would be given by $\overline{k_v}[x],$ contradicting the fact that we previously argued this cotangent sheaf was zero. Hence, we must be in the former case. In particular, this morphism is flat. Combining this with the semistable/good reduction argument above lets us conclude that $\mathcal{E}' \rightarrow \mathcal{E}$ is flat.
\par We now check finiteness. By Zariski's main theorem (in the form \cite[\href{https://stacks.math.columbia.edu/tag/02LR}{Tag 02LR}]{stacks-project}), since we know that $f: \mathcal{E}' \rightarrow \mathcal{E}$ is quasifinite and separated, we know that $\mathcal{E}'$ is an open subscheme of an integral morphism $\overline{f}: \mathcal{E}'' \rightarrow \mathcal{E},$ given by taking the normalization of $\mathcal{E}$ in $\mathcal{E}' \rightarrow \mathcal{E}.$ Note that $\overline{f}$ is finite by \cite[\href{https://stacks.math.columbia.edu/tag/0AVK}{Tag 0AVK}]{stacks-project}. Furthermore, $\mathcal{E}''$ is normal and integral, by \cite[\href{https://stacks.math.columbia.edu/tag/035L}{Tag 035L}]{stacks-project} (since $\mathcal{E}'$ satisfies these properties).
\par Our goal will be to show that $\mathcal{E}' = \mathcal{E}''.$ We will do this by showing that $\overline{f}$ is finite \'{e}tale, and therefore also the N\'{e}ron model of its generic fiber by Proposition \ref{prop: finite etale covers are neron models}, which will show that $\mathcal{E}' = \mathcal{E}''$ and thus prove the claim.
\par We first argue that over every point of codimension one in $\mathcal{E},$ the fiber of $\overline{f}$ over that point agrees with the corresponding fiber for $f.$ To that end, suppose $R$ is the local ring of a codimension one point in $\mathcal{E};$ as a regular Noetherian local ring of dimension one, this is a DVR. We know that its pre-image in $\mathcal{E}'',$ since $\mathcal{E}''$ is finite, is equal to $\Spec S$ for $S$ a finite $R$-algebra, and the pre-image of $\Spec R$ under $f$ is some open subset of $\Spec S;$ call this open subset $U.$ 
\par We first prove that $S$ is torsion-free over $R.$ To show this, note that $R$ is the localization of the ring of functions for some affine open in $\mathcal{E};$ the same is therefore true for $S$ in $\mathcal{E}''.$ In particular, as $\mathcal{E}''$ is integral, $S$ is an integral domain, whose generic point can be identified with that of $\mathcal{E}''.$ Furthermore, $R \rightarrow S$ is an injection since the generic point of $\Spec S$ maps to the generic point of $\Spec R$ (since this holds for $\overline{f}$). Hence, by the PID structure theorem, as an $R$-module, $S$ is free over $R.$ This means, in particular, that the degree of $f$ over our codimension one point is equal the degree of $f$ over the generic point, which is $n.$ 
\par We now do casework on where this codimension one point lies. First, suppose this point is in the generic fiber of $\mathcal{E}.$ Then, we already know that $\mathcal{E}'$ contains the generic fiber of $\mathcal{E}''.$ In particular, $\Spec S \rightarrow \mathcal{E}''$ factors through $\mathcal{E}'.$ 
\par Next, suppose this point lies in the fiber over a place $v$ of $K$ where $\mathcal{E}$ has good reduction. Let $K_v$ be the completion of $K$ with respect to the place $v,$ with valuation ring $\mathcal{O}_{K_v}.$ In this case, $\mathcal{E}'_{\mathcal{O}_{K_v}}, \mathcal{E}_{\mathcal{O}_{K_v}}$ are abelian schemes and hence proper over $\mathcal{O}_{K_v}.$ Then, $\mathcal{E}' \rightarrow \mathcal{E}$ is finite \'{e}tale of degree $n.$ Thus, the pre-image of $\Spec R$ under this map is of the form $\Spec S',$ where $\Spec S' \rightarrow \Spec S$ is an open immersion (as the pullback of $\mathcal{E}' \rightarrow\mathcal{E}''$) and finite. But this means that this is an isomorphism, so $\mathcal{E}'$ contains $\Spec S$ in this case also.
\par Finally, suppose that this codimension one point lies over a place $v$ of bad reduction of $\mathcal{E},$ with residue field $k_v$ and an algebraic closure $\overline{k_v}.$ Passing to the geometric special fiber of $\mathcal{E}' \rightarrow \mathcal{E},$ we know that this morphism is a map from $nk$ components to $k$ components for some positive integer $k,$ where the components are either all $\mathbb{G}_m$s or all $\mathbb{G}_a$s (depending on the type of bad reduction). Since the closed point of $R$ is codimension one, its closed point (when we then base change to $\overline{k_v}$) corresponds to some $G_{\overline{k_v}/k_v}$-orbit of generic points of these components. 
\par Note that $\mathcal{E}' \rightarrow \mathcal{E}$ is quasifinite (as it is \'{e}tale and quasicompact), so the pre-image of a generic point of one of the components in $\mathcal{E}_{\overline{k_v}}$ is some number of generic points of components in $\mathcal{E}'_{\overline{k_v}}.$ So suppose some generic point of some component is not in the image of $f_{\overline{k_v}}: \mathcal{E}'_{\overline{k_v}} \rightarrow \mathcal{E}_{\overline{k_v}}.$ Then, the pre-image of a generic point of one of these components under $\mathcal{E}' \rightarrow \mathcal{E}$ has more than $n$ points by pigeonhole, so the same must hold for $\mathcal{E}'' \rightarrow \mathcal{E}.$ In particular, the degree of $\mathcal{E}''_{\overline{k_v}} \rightarrow \mathcal{E}_{\overline{k_v}} $ at this point is more than $n.$ However, we proved above that the degree of $\mathcal{E}'' \rightarrow \mathcal{E}$ at any codimension one point is $n,$ which is a contradiction.
\par Therefore, each geometric component lies in the image of $f_{\overline{k_v}}.$ By translating by any point in any one of these components, the number of points in the pre-image of the generic point of each of these components in $\mathcal{E}_{\overline{k_v}}$ must be the same. Thus, the pre-image of a generic point of one of these components in $\mathcal{E}_{\overline{k_v}}$ must consist of the generic points of exactly $n$ of the components in $\mathcal{E}'_{\overline{k_v}}.$ By the same degree considerations as above, we see that these points have to be the pre-image of the generic point of components under $\mathcal{E}'' \rightarrow \mathcal{E}$ as well. Therefore, it follows that the open subscheme $\mathcal{E}'$ includes every point of $\mathcal{E}''$ lying over a codimension one point of $\mathcal{E}.$ In particular, $f$ is \'{e}tale in codimension one, since $\mathcal{E}'\rightarrow \mathcal{E}$ is \'{e}tale.
\par Finally, normality of $\mathcal{E}''$ tells us that, by purity of the branch locus (see \cite[\href{https://stacks.math.columbia.edu/tag/0BMB}{Tag 0BMB}]{stacks-project}), $\mathcal{E}'' \rightarrow \mathcal{E}$ is \'{e}tale, and thus finite \'{e}tale, which we argued is enough to prove the proposition.
\end{proof}
\subsection{Uniformity in the Elliptic Curve Case}
We conclude this section by showing that we have a uniformity statement for the size of $\pietgeo(\mathcal{E}),$ strenghtening Theorem \ref{thm: basic finiteness} in the case of elliptic curves. Recall the following theorem from the introduction.
\uniformity*
The fact that $E$ is an elliptic curve, rather than an abelian variety, is used in two crucial ways. One of these is in the use of Merel's torsion theorem from \cite{merel}, which gives a bound for the size of the torsion subgroup of $E(K)$ for any elliptic curve $E/K$ only depending on $[K:\mathbb{Q}].$ The other uses the fact that $E[n](\overline{K}) \simeq (\mathbb{Z}/n)^2.$ 
\begin{proof}
Let $s: \Spec \mathcal{O}_K \rightarrow \mathcal{E}$ be the unit section of $\mathcal{E}.$ We already know that $\pietgeo(\mathcal{E})$ is finite by Theorem \ref{thm: basic finiteness}, so we have a finite \'{e}tale cover of $\mathcal{E}$ corresponding to the $\piet$-set $\piet(\mathcal{E})/s_*\piet(\Spec \mathcal{O}_K);$ let this cover be  $\mathcal{E}' \rightarrow \mathcal{E}.$ We know that $\mathcal{E}'$ has a $\Spec \mathcal{O}_K$-section by considering the pullback of this cover along $s.$ By Proposition \ref{prop: finite etale covers are neron models}, we know that $\mathcal{E}'$ is the N\'{e}ron model of its generic fiber $E',$ with $E' \rightarrow E$ an isogeny by Lemma \ref{lem: generic fiber finite etale covers are abelian}. 
\par Now, consider the group structure of the $\overline{K}$-points of the kernel of $E' \rightarrow E.$ If this isn't cyclic, then by the structure theorem for finitely generated abelian groups, it follows that there is some $n \geq 2$ such that $(\mathbb{Z}/n)^2$ is a subgroup of the kernel. In other words, the kernel of $E' \rightarrow E$ contains $E'[n],$ which by Lemma \ref{lem: multiplication by N not etale} contradicts the assumption that $\mathcal{E}' \rightarrow \mathcal{E}$ is \'{e}tale. In other words, the $\overline{K}$-points of the kernel form a cyclic group, say $\mathbb{Z}/N.$
\par We now recall the Galois structure, arising from the fact that the kernel is a group scheme over $K.$ Note that the pre-image of the unit section $\Spec \mathcal{O}_K \rightarrow \mathcal{E}$ is a finite \'{e}tale cover of $\Spec \mathcal{O}_K,$ which is therefore a disjoint union of $\Spec \mathcal{O}_{K'}$ for varying finite extensions $K'/K$ which are unramified everywhere. In particular, the points of the kernel are defined over an everywhere unramified extension of $K.$
\par Now, since the kernel of $E' \rightarrow E$ is a group scheme over $K,$ we can view it as $\mathbb{Z}/N$ equipped with a $G_K$-action which is compatible with the group law. This gives us a morphism \[G_K \rightarrow (\mathbb{Z}/N)^*.\] But $(\mathbb{Z}/N)^*$ is abelian, so hence this factors through the Galois group of the maximal abelian extension of $K$ over $K.$ Furthermore, by the fact that all the points of the kernel are defined over an unramified extension, this further factors through the Galois group of $K^{hcf}/K,$ where $K^{hcf}$ is the Hilbert class field of $K.$ But $K^{hcf}$ is a number field (since the class group of $K$ is finite), meaning that all the points in the kernel of $E' \rightarrow E$ are defined over the number field $K^{hcf}.$
\par Finally, by Merel's theorem (\cite[Corollary to Theorem on p.1]{merel}), there is a uniform bound for the number of torsion points defined over $K^{hcf}$ of any elliptic curve over $K^{hcf}.$ In particular, this bounds the possible size of the kernel of $E' \rightarrow E,$ and therefore its degree. However, this degree was constructed to be the size of $\pietgeo(\mathcal{E}),$ and so we obtain a uniform bound for the possible size of $\pietgeo(\mathcal{E}),$ which is what we wanted to show. 
\end{proof}
If $E$ has a place of additive reduction, we can make this bound explicit and independent of $K.$
\begin{prop}\label{prop: strong bounds for additive reduction case}
Suppose that $E/K$ is an elliptic curve with at least one place of additive reduction, and $\mathcal{E}$ its N\'{e}ron model over $\mathcal{O}_K.$ Then, $|\pietgeo(\mathcal{E})| \leq 4.$
\end{prop}
The idea is to use a similar logic to that of Lemma \ref{lem: component grows under finite etale}; a finite \'{e}tale cover would have to have more components than $\mathcal{E}.$ However, in the additive reduction case, the number of such components is bounded.
\begin{proof}
Suppose that $\mathcal{E}' \rightarrow \mathcal{E}$ is a connected finite \'{e}tale cover, where $\mathcal{E}'$ has a $\Spec \mathcal{O}_K$ point. Then, Lemma \ref{lem: generic fiber finite etale covers are abelian} and Proposition \ref{prop: finite etale covers are neron models} tells us that $\mathcal{E}'$ is the N\'{e}ron model of its generic fiber $E'.$ Furthermore, we know from \cite[\S 7.3, Corollary 7]{BLRNeron} that $E'$ has additive reduction at the same places that $E$ does. Let $v$ be this place, with $\overline{k_v}$ an algebraic closure of the corresponding residue field. 
\par Now, from the Kodaira-N\'{e}ron classification (again, see \cite[Chapter IV, Theorem 8.2]{advAEC}) and the fact that the N\'{e}ron model is the smooth locus of the minimal regular model (see \cite[Chapter IV, Theorem 6.1]{advAEC}), we see that the base changes $\mathcal{E}_{\overline{k_v}}', \mathcal{E}_{\overline{k_v}}$ of these schemes to $\overline{k_v}$ consists of at most $4$ $\mathbb{G}_a$s. Furthermore, the morphism between them is finite \'{e}tale by assumption.
\par Now, we can use Lemma \ref{lem: additive reduction iso or crushes to point}. Note that $\mathcal{E}'_{\overline{k_v}} \rightarrow \mathcal{E}_{\overline{k_v}}$ is a map from $4$ of disjoint copies of $\mathbb{G}_a$ to at most $4$ disjoint copies of $\mathbb{G}_a,$ so therefore each of the copies of $\mathbb{G}_a$ in the former are either mapped isomorphically, or crushed to a point. But the latter cannot happen if this morphism is also finite \'{e}tale. Hence, this morphism is an isomorphism when restricted to any component of the source (and its image in the target).
\par Therefore, the pre-image of any point consists of points which all lie on different components. It follows that the degree of $\mathcal{E}'_{\overline{k_v}} \rightarrow \mathcal{E}_{\overline{k_v}}$ is at most $4.$ But this means that $\mathcal{E}' \rightarrow \mathcal{E}$ is also degree at most $4.$ This is enough to tell us that $|\pietgeo(\mathcal{E})| \leq 4,$ as desired. 
\end{proof}
\section{Elliptic Curves over $\mathbb{Q}$}
We finish the paper by making explicit what the possible groups are for $\pietgeo(\mathcal{E})$ in the case where we have elliptic curves over $\mathbb{Q}.$ In our setting,  $\piet(\mathcal{E}) = \pietgeo(\mathcal{E}),$ as $\piet(\Spec\mathbb{Z})$ is trivial. Our goal in this section is to prove the following theorem from the introduction.
\classificationQ*
Most of this section will be concerned in the case where we have semistable reduction everywhere. We will start by showing that the above list is exhaustive; that is, no other groups are possible. Once we do that, we will give examples of curves which attain each \'{e}tale fundamental group, which were found using the database of elliptic curves over $\mathbb{Q}$ from \cite{lmfdb}. We will conclude this section with the proof in the case when we have additive reduction at some place.
\subsection{Proof of Non-Existence}
To begin, suppose that $E/\mathbb{Q}$ is an elliptic curve with semistable reduction everywhere, with N\'{e}ron model $\mathcal{E}.$ Suppose that $f: \mathcal{E}' \rightarrow \mathcal{E}$ is a finite \'{e}tale cover of the N\'{e}ron model of $E$ with connected generic fiber. Then, notice that the kernel of this morphism is finite \'{e}tale over $\Spec \mathbb{Z}.$ But $\piet(\Spec \mathbb{Z}) = 1,$ so finite \'{e}tale covers of $\Spec \mathbb{Z}$ are disjoint unions of $\Spec \mathbb{Z}.$ Therefore, if we pass to the generic fiber $f_{\mathbb{Q}}: E' \rightarrow E$ of $f,$ we see that $E'$ is an elliptic curve over $\mathbb{Q},$ and the kernel of $f_{\mathbb{Q}}$ consists of points which are all defined over $\mathbb{Q}.$
\par We now cite Mazur's torsion theorem, which reduces the proof of the above theorem to a finite list of groups to analyze.
\begin{thm}[Theorem 8 from \cite{mazurtorsion}]\label{citedthm: mazur torsion}
Let $E/\mathbb{Q}$ be an elliptic curve. Then, the torsion subgroup of $E(\mathbb{Q})$ is one of the following subgroups:
\begin{align*}
    \mathbb{Z}/N\mathbb{Z} & \text{ for } 1 \leq N \leq 12, N \neq 11 \\ \mathbb{Z}/2\mathbb{Z} \times \mathbb{Z}/2N\mathbb{Z} & \text{ for } 1 \leq N \leq 4.
\end{align*}
\end{thm}
We start by observing that any of the groups in the second line are not possible for the kernel of the isogeny $\mathcal{E}' \rightarrow \mathcal{E}$ if this isogeny is finite \'{e}tale, by Lemma \ref{lem: multiplication by N not etale}. 
\par We now prove the following intermediate claim.
\begin{prop}\label{prop: finite etale covers are prime}
In the above setting, the degree of $f$ must be prime or one.
\end{prop}
\begin{proof}
Our starting place is the following formula, which comes from taking Silverman's formula, \cite[Proposition 1.1]{silvermanheights}, over $\mathbb{Q}:$
\[h(E) = \frac{1}{12} \left(\log |\Delta_{E/\mathbb{Q}}| - \log |\Delta(\tau) \text{Im} (\tau)^6|\right).\] Here, $\Delta_{E/\mathbb{Q}}$ is the minimal discriminant of $E,$ $\Delta$ is the discriminant modular form on $\mathbb{H},$ the complex upper half plane, and $\tau \in \mathbb{H}$ is the element in the standard fundamental domain such that $\mathbb{C}/(\mathbb{Z} + \mathbb{Z}\tau)$ is isomorphic as complex analytic spaces to $E_{\mathbb{C}}.$ Let $\Delta_{E'/\mathbb{Q}}, \tau'$ be the corresponding quantities for $E'.$
\par We first show that, given $E,$ there is only at most one possible prime degree for $f,$ and we give a condition this prime needs to satisfy. Hence, suppose that the degree of our finite \'{e}tale cover is $p.$ By Lemma \ref{lem: component grows under finite etale}, we know that $\Delta_{E'/\mathbb{Q}} = \Delta_{E/\mathbb{Q}}^p.$ On the other hand, we know from the complex analytic description of isogeneous elliptic curves that the imaginary part of $\tau'$ is either $p$ or $1/p$ times that of $\tau.$
\par Now, we can estimate the size of $\Delta(\tau);$ for this, we use \cite[Exercise in \S 2]{silvermanheights}. Namely, if we set \[c_{\tau} = \log |\Delta(\tau)| + 12 \log 2\pi + 2 \pi \text{Im} \tau,\] then $|c_{\tau}| \leq \frac{1}{9}.$ 
\par We now put these together, using the formulas for both $E, E'.$ Taking the difference $h(E') - h(E),$ Faltings's isogeny formula (\cite[Lemma 5]{faltingsmordell}) tells us that this is $-\frac{1}{2} \log p.$ Combining this with Silverman's formula, we see that \[-\frac{1}{2}\log p = \frac{p-1}{12} \log |\Delta_{E/\mathbb{Q}}| + \frac{c_{\tau} - c_{\tau'}}{12} - \frac{\pi}{6} \text{Im} \tau + \frac{\pi}{6} \text{Im} \tau' + \frac{1}{2}\log \left|\frac{\text{Im} \tau}{\text{Im} \tau'}\right|.\] But we know that the left-hand side is negative. Furthermore, we know that $\text{Im} \tau \geq \frac{\sqrt{3}}{2},$ and $|c_{\tau} - c_{\tau'}| \leq 2/9.$ So hence the imaginary part of $\tau$ has to be larger than that of $\tau'$ for the right-hand side to be negative, so $\text{Im} \tau' = \frac{1}{p}\text{Im} \tau.$ We can then rewrite this as \[-\frac{1}{2}\log p = \frac{p-1}{12} \log |\Delta_{E/\mathbb{Q}}| + \frac{c_{\tau} - c_{\tau'}}{12} - \frac{(p-1)\pi}{6p} \text{Im} \tau  + \frac{1}{2}\log p,\] or that \[-\log p = \frac{p-1}{12} \log |\Delta_{E/\mathbb{Q}}| + \frac{c_{\tau} - c_{\tau'}}{12} - \frac{(p-1)\pi}{6p} \text{Im} \tau.\] In other words, we see that, since $p - 1 \geq 1$ and $c_{\tau}, c_{\tau}'$ have magnitudes are at most $1/9,$ that \[|\frac{\pi}{p}\text{Im}\tau - \frac{1}{2}\log |\Delta_{E/\mathbb{Q}}| - \frac{6\log p}{p-1}| \leq \frac{1}{9}.\] 
\par Now, suppose for the sake of contradiction that there existed a finite \'{e}tale cover $f: \mathcal{E}' \rightarrow \mathcal{E}$ whose degree was not prime. We've argued previously that the kernel of the isogeny consists of $\mathbb{Q}$-points, so the isogeny can be factored into isogenies of prime degree. Say the last two of these isogenies are $E_2 \rightarrow E_1 \rightarrow E,$ which spread out to morphisms of N\'{e}ron models $\mathcal{E}_2 \rightarrow \mathcal{E}_1 \rightarrow \mathcal{E}.$ We know from Lemma \ref{lem: finite etale factors into finite etales} that these are all finite \'{e}tale. Say the degree of the first map is $p_2,$ the second is $p_1.$
\par Our above arguments show that $\text{Im} \tau_1 = \frac{1}{p_1} \text{Im} \tau, \text{Im} \tau_2 = \frac{1}{p_2} \text{Im} \tau_1.$ Therefore, we require that \[|\frac{\pi}{p_1}\text{Im}\tau - \frac{1}{2}\log |\Delta_{E/\mathbb{Q}}| - \frac{6\log p_1}{p_1 - 1}| \leq \frac{1}{9}\] and \[|\frac{\pi}{p_1p_2}\text{Im}\tau - \frac{p_1}{2}\log |\Delta_{E/\mathbb{Q}}| - \frac{6 \log p_2}{p_2 - 1}| \leq \frac{1}{9}.\] In other words, we require that \[\left|\frac{(p_2 - 1)\pi}{p_1p_2}\text{Im}\tau + \frac{p_1 - 1}{2}\log |\Delta_{E/\mathbb{Q}}| + 6\log \frac{p_2^{1/(p_2 - 1)}}{p_1^{1/(p_1 - 1)}}\right| \leq \frac{2}{9}.\] Now, the last term is the only possible one which might not be positive, so this implies that \[ 0 \leq \frac{(p_2 - 1)\pi}{p_1p_2}\text{Im}\tau + \frac{p_1 - 1}{2}\log |\Delta_{E/\mathbb{Q}}| \leq \frac{2}{9} + 6\max\{0, \log \frac{p_1^{1/(p_1 - 1)}}{p_2^{1/(p_2 - 1)}}\}.\] Now, as we are working over $\mathbb{Q},$ Mazur's torsion tells us that $2 \leq p_1, p_2 \leq 7.$ We now do some casework.
\par If $p_1 = 7,$ then we find that $\log |\Delta_{E/\mathbb{Q}}| \leq \frac{2}{27} + 2\log 7^{1/6},$ or that $\Delta_{E/\mathbb{Q}} < 3.$ If $p_1 = 5,$ we get the bound $1/9 + 3\log 5^{1/4},$ or $|\Delta_{E/\mathbb{Q}}| < 4.$ Next, if $p_1 = 3,$ we get the bound $2/9 + 6 \log 3^{1/2}$ for $\log |\Delta_{E/\mathbb{Q}}|,$ or $|\Delta_{E/\mathbb{Q}}| < 34.$ Finally, say that $p_1 = 2.$ Then, we know that $\frac{\pi(p_2 - 1)}{p_2} \text{Im} \tau + \log |\Delta_{E/\mathbb{Q}}| \leq 4/9 + 12 \log 2,$ which as $p_2 \geq 2$ means that \[\frac{\pi \text{Im} \tau}{2} + \log \Delta_{E/\mathbb{Q}} \leq 4/9 + 12 \log 2.\]
But we also know that \[\frac{\pi}{2}\text{Im}\tau - \frac{1}{2}\log |\Delta_{E/\mathbb{Q}}| - 6\log 2 \geq -\frac{1}{9},\] which implies that \[\frac{3}{2} \log |\Delta_{E/\mathbb{Q}}| \leq \frac{1}{3} + 6 \log 2,\] or that $|\Delta_{E/\mathbb{Q}}| < 90.$
However, this means that the conductor of $E,$ and thus of $E_1, E_2,$ is also less than $90,$ (e.g. see \cite[Chapter IV, Corollary 11.2]{advAEC}). In other words, if such an isogeny existed which wasn't prime, then it would come from an isogeny class of elliptic curves with semistable reduction everywhere with conductor less than $90.$ However, one can then check by a search through \cite{lmfdb} that no such isogenies exist: either Lemma \ref{lem: component grows under finite etale} fails for those isogenies, or the change in Faltings height is not the correct one. 
\end{proof}
Hence, we have already shown that the only possible values for the group $\piet(\mathcal{E})$ are among the groups $\mathbb{Z}/N\mathbb{Z}$ for $N = 1, 2, 3, 5, 7.$ We now need to show that $7$ cannot be attained.
\subsection{Ruling out $N = 7$: the General Strategy}
To rule out the group $\mathbb{Z}/7\mathbb{Z},$ we prove the following stronger claim.
\begin{prop}\label{prop: no seventh power discriminant and torsion}
There does not exist an elliptic curve $E/\mathbb{Q}$ with semistable reduction everywhere, whose minimal discriminant is a seventh power, and which has a $7$-torsion point defined over $\mathbb{Q}.$
\end{prop}
We first verify that this proposition implies that $\mathbb{Z}/7\mathbb{Z}$ is not the \'{e}tale fundamental group of the N\'{e}ron model of an elliptic curve $E/\mathbb{Q}$ with semistable reduction everywhere. To that end, suppose for the sake of contradiction that such an $E$ existed, and $\mathcal{E}$ its N\'{e}ron model over $\mathbb{Z}.$ Then, we have a degree $7$ finite \'{e}tale cover $\mathcal{E}' \rightarrow \mathcal{E}.$ 
\par Let $E'$ be the generic fiber of $\mathcal{E}';$ this is an elliptic curve over $\mathbb{Q}$ by Lemma \ref{lem: generic fiber finite etale covers are abelian}. We know that the kernel of $E' \rightarrow E$ consists of rational points, so hence $E'$ has a $7$-torsion point over $\mathbb{Q}.$ Furthermore, it has semistable reduction everywhere since this property is preserved under isogeny (e.g. see \cite[\S 7.3, Corollary 7]{BLRNeron}). Finally, the minimal discriminant of $E'$ is a seventh power by Lemma \ref{lem: component grows under finite etale}. This contradicts the above proposition.
\par Hence, we turn to proving Proposition \ref{prop: no seventh power discriminant and torsion}. Assume for the sake of contradiction that there existed an elliptic curve $E/\mathbb{Q}$ with semistable reduction, such that it has a rational $7$-torsion point and such that its minimal discriminant is a seventh power. We start by proving the following lemma. 
\begin{lemma}\label{lem: unramified 7-torsion}
Suppose that $E/\mathbb{Q}$ is an elliptic curve with a $7$-torsion point. Let $p \neq 7$ be a prime. Suppose that $E$ has semistable reduction at $p$ and that $v_p(\Delta)$ is divisible by $7.$ Then, $E[7],$ as a $G_{\mathbb{Q}_p}$-module, is unramified. 
\end{lemma}
\begin{proof}
First, if $v_p(\Delta) = 0,$ this follows from the N\'{e}ron-Ogg-Shafarevich criterion. So suppose that $E$ has semistable bad reduction. Let $\mathcal{E}$ be the N\'{e}ron model of $E_{\mathbb{Q}_p}$ over $\mathbb{Z}_p.$ We know that, since $p \neq 7,$ the multiplication by $7$ map is an \'{e}tale isogeny of the N\'{e}ron model (e.g. see \cite[\S 7.3, Lemma 2]{BLRNeron}), and so is a quasifinite \'{e}tale morphism. If we look at the kernel, this is a quasifinite \'{e}tale group scheme over $\Spec \mathbb{Z}_p.$
\par The structure of quasifinite group schemes over a henselian base (for instance, see \cite[\S 7.3, p.179]{BLRNeron})
tells us that we can write $\mathcal{E}[7]$ as the disjoint union of a finite group scheme $X_f$ and another scheme which has no special fiber, $X_{\eta}.$ Note that $X_f$ is finite \'{e}tale, since $X_f$ is an open subscheme of $\mathcal{E}[7],$ which we already know is an \'{e}tale group scheme over $\mathbb{Z}_p.$
\par Now, consider the degree of $X_f \rightarrow \Spec \mathbb{Z}_p.$ We can compute this by looking at the special fiber and base changing to the algebraic closure of $\mathbb{F}_p.$ Since $7|v(\Delta),$ as mentioned in the proof of Lemma \ref{lem: component grows under finite etale}, the geometric special fiber of $\mathcal{E}$ consists of $7k$ copies of $\mathbb{G}_m$ for some positive integer $k.$ As the component group is cyclic, as noted in \cite[Chapter IV, \S 9]{advAEC}, there are $7$ components of this geometric special fiber of $\mathcal{E}$ which, under multiplication by $7,$ get sent to the identity component. Furthermore, each of these $7$ components has $7$ points which are sent to the identity by multiplication by $7,$ since this is true for the identity component and since multiplication by $7$ is surjective on the identity component. Therefore, the degree of $X_f \rightarrow \Spec \mathbb{Z}_p$ is $49.$
\par On the other hand, we know the degree of the generic fiber pf $\mathcal{E}[7]$ is $49 = |E[7](\overline{K})|,$ and this degree is at least the degree of $X_f$ (since the generic fiber also has the $X_{\eta}$ part). Therefore, these are actually equal, so $\mathcal{E}[7]$ is equal to its finite part. In particular, this means that all the $7$-torsion points are defined over unramified field extensions of $\mathbb{Q}_p,$ as $X_f = \mathcal{E}[7]$ being finite \'{e}tale over $\mathbb{Z}_p$ means it is a disjoint union of valuation rings of unramified extensions of $\mathbb{Q}_p.$ In particular, this means that the action of inertia is trivial on the $7$-torsion points, which is what we wanted to show.
\end{proof}
We now encode the condition of having seventh power discriminant in a more concrete form, putting us in a situation similar to that studied in \cite{pss}.
\begin{prop}\label{prop: seventh power discriminant equation}
Suppose that an elliptic curve $E/\mathbb{Q}$ with semistable reduction everywhere has minimal discriminant which is a seventh power. Then, there exist relatively prime integers $A, B, C$ such that $A^2 + B^3 = 1728C^7,$ and such that $E$ has $j$-invariant $B^3/C^7.$
\end{prop}
\begin{proof}
As $\mathbb{Q}$ has class number one, we can find a minimal Weierstrass equation of the form \[y^2 + a_1xy + a_3 y = x^3 + a_2x^2 + a_4x + a_6,\] where the $a_i$ are integers. Now, we can use the standard change of variables, as in \cite[\S III.1]{AEC}: we have expressions $c_4, c_6$ which are integer polynomials in the $a_i$ such that $1728\Delta = c_4^3 - c_6^2.$ By assumption, we know that $\Delta = -C^7$ for some integer $C.$ Furthermore, by the semistability assumption, either $p$ doesn't divide $\Delta,$ or $p$ doesn't divide $c_4$ (from the description of Tate's algorithm, \cite[Chapter IV, \S 9]{advAEC} and the formula for $c_4$ in \cite[\S III.1]{AEC}). Either way, $\Delta, c_4$ are relatively prime. 
\par Hence, setting $A = c_6, B = -c_4,$ we find that we have our relatively prime integers $A, B, C$ satisfying the above equation. Furthermore, our elliptic curve $E$ has $j$-invariant $B^3/C^7$ from the formula for the $j$-invariant (see \cite[\S III.1]{AEC}, for instance), which gives the desired claim.
\end{proof}
Hence, the $j$-invariant of an elliptic curve $E/\mathbb{Q}$ whose minimal discriminant is seventh power is the image of a $\mathbb{Z}$-point under the morphism \[S = \Spec \mathbb{Z}[A, B, C]/(A^2 + B^3 = 1728C^7) - \{A = B = C =0\} \xrightarrow{j = B^3/C^7} \mathbb{P}^1_{\mathbb{Z}},\] where, $\{A = B = C = 0\}$ is just the closed $\mathbb{Z}$-subscheme cut out by these equations. Note that we can view $\mathbb{P}^1_{\mathbb{Z}}$ as an integral model for the modular curve $X(1) \simeq \mathbb{P}^1_{\mathbb{Q}}.$ 
\par On the other hand, we know that, to have a $7$-torsion point over $\mathbb{Q},$ our point lies in the image of a $\mathbb{Q}$-point under the map $X_1(7) \rightarrow X(1)$ (by definition of $X_1(7)$). Hence, if there existed an elliptic curve $E$ is such that its N\'{e}ron model has $\piet$ equal to $\mathbb{Z}/7\mathbb{Z},$ then we have a rational point on $X(1)$ which lies in the image of both maps, and which is not $\infty$ (since the rational point needs to be the $j$-invariant of the elliptic curve).
\par We now draw from the techniques in \cite{pss}, specifically in the construction of the appropriate curves needed to be considered and their local test in \cite[\S 7.4]{pss}. We will start by interpreting our elliptic curve $E$ with a rational $7$-torsion point as a rational point in the image of a map $X_E \rightarrow X(1)$ for some appropriate twist of the map of modular curves $X(7) \rightarrow X(1)$ by certain elements of Galois cohomology, following \cite[\S 4-5]{pss}. Using this, we can use the explicit description of the morphism $X_E \rightarrow X(1)$ to allow us to describe which points in $X(1)$ lies in the image of this morphism, using $p$-adic information for an appropriate prime $p.$ But we will then show that this locus is disjoint from the possible image of $S\rightarrow \mathbb{P}^1,$ given by $j = B^3/C^7.$ 
\subsection{Constructing the Relevant Twists}
\par We start by describing these curves $X_E,$ which as described in \cite[\S 4.4-4.5]{pss} are twists $X(7) \rightarrow X(1).$ Recall that $X(7)$ is a modular curve defined over $\mathbb{Q},$ whose $K$-points parametrize elliptic curves $E/K$ and a $G_K$-equivariant isomorphism $\phi: E[7](\overline{k}) \simeq \mu_7 \times \mathbb{Z}/7$ which preserves the symplectic structure on both sides (described in \cite[\S 4.1]{pss}).
\par By viewing $\sll_2(\mathbb{F}_7)$ as the $\overline{\mathbb{F}_7}$ symplectic automorphisms of $\mu_7 \times \mathbb{Z}/7,$ we have an action of $G_{\mathbb{Q}}$ on $\sll_2(\mathbb{F}_7).$ We then can consider the the cohomology set $H^1(G_{\mathbb{Q}}, \sll_2(\mathbb{F}_7);$ elements of this set are in bijection with twists of the $G_{\mathbb{Q}}$-module $\mu_7 \times \mathbb{Z}/7.$ 
\par These twists then induce twists of $X(7),$ which for an element $\xi \in H^1(\mathbb{Q}, \text{SL}_2(\mathbb{F}_7))$ is the curve $X_{\xi}$ which parametrizes elliptic curves $E/K$ and a $G_K$-equivariant symplectic isomorphism from $E[7](\overline{k})$ to the twist of the module $\mu_7 \times \mathbb{Z}/7$ corresponding to $\xi.$ Like $X(7),$ these are planar quartic curves.
\par These twists of $X(7)$ can also be viewed as elements of $H^1(\mathbb{Q}, \text{PSL}_2(\mathbb{F}_7))$ where the Galois action arises by viewing $\text{PSL}_2(\mathbb{F}_7)$ as the automorphisms of $X(7)_{\overline{\mathbb{Q}}}.$ This Galois action is compatible with that of $\sll_2(\mathbb{F}_7);$ see \cite[\S 4.4]{pss}. Furthermore, given a twist of $\mu_7 \times \mathbb{Z}/7,$ viewed as a cocycle $\xi \in H^1(\mathbb{Q}, \sll_2(\mathbb{F}_7)),$ the corresponding twist of $X(7)$ is the image of $\xi$ under \[H^1(\mathbb{Q}, \text{SL}_2(\mathbb{F}_7))\rightarrow H^1(\mathbb{Q}, \text{PSL}_2(\mathbb{F}_7))\] with the aforementioned Galois actions.
\par In our setting, we want the twists of $\mu_7 \times \mathbb{Z}/7$ which are the $7$-torsion of a given elliptic curve $F;$ we denote the corresponding curve by $X_F.$ Specifically, we are interested in elliptic curves with a rational $7$-torsion point and whose discriminant is a seventh power. By Lemma \ref{lem: unramified 7-torsion}, since $\mu_7 \times \mathbb{Z}/7$ is unramified outside of $7,$ the cohomology classes $\xi$ that we twist by are unramified outside of $7.$
\par We further narrow down the possibilites for our cohomology class. Our elliptic curve $E$ was assumed to have a rational $7$-torsion point; therefore, $E[7](\overline{\mathbb{Q}})$ contains $\mathbb{Z}/7\mathbb{Z}$ as a Galois submodule. Following \cite[\S 5.3]{pss}, this means that the cohomology class in $H^1(\mathbb{Q}, \text{PSL}_2(\mathbb{F}_7))$ we use to twist $X(7)$ into $X_E(7)$ must come from $H^1(\mathbb{Q}, B) \rightarrow H^1(\mathbb{Q}, \text{PSL}_2(\mathbb{F}_7)),$ where $B$ is the image under $\sll_2(\mathbb{F}_7) \rightarrow \text{PSL}_2(\mathbb{F}_7)$ of the group of upper triangular matrices in $\sll_2(\mathbb{F}_7).$ To see this, fix a group isomorphism $\phi: \mu_7 \times \mathbb{Z}/7\mathbb{Z} \rightarrow E[7](\overline{\mathbb{Q}})$ which preserves the symplectic structure and sends $\mu_7$ to $\mathbb{Z}/7\mathbb{Z}.$ Then, the cohomology class twisting $\mu_7 \times \mathbb{Z}/7$ into $E[7]$ is represented by the cocycle $\sigma \mapsto \phi^{-1}\circ \phi^{\sigma}.$ Furthermore, for each $\sigma,$ this sends $\mu_7$ to $\mu_7,$ so is valued in upper-triangular matrices, and thus mapping this through $\sll_2(\mathbb{F}_7) \rightarrow \text{PSL}_2(\mathbb{F}_7)$ yields a $B$-valued cocycle.
\par Now, consider the map $H^1(\mathbb{Q}, B) \rightarrow H^1(\mathbb{Q}, \mathbb{Z}/3\mathbb{Z}),$ viewing $\mathbb{Z}/3$ as the quotient of $B$ by the unipotent upper triangular matrices (whose induced Galois action is trivial). We can also view this as being the restriction of our cocycle to $\mu_7.$ Since we know that $\phi$ sends $\mu_7$ to $\mathbb{Z}/7,$ our cocycle gets sent to the cocycle in $H^1(\mathbb{Q}, \mathbb{Z}/3\mathbb{Z}) \simeq \text{Hom}(G_{\mathbb{Q}}, \mathbb{Z}/3\mathbb{Z})$ given by the homomorphism $f$ sending $\sigma$ to the element \[a_{\sigma} \in \mathbb{F}_7^{\times}/\{\pm 1\} \simeq \text{Aut}(\mu_7)/\{\pm 1\},\] such that \[\zeta_7^{\pm a_{\sigma}} = \sigma^{-1}(\zeta_7)^{\pm 1}.\] This is because if the isomorphism $\phi$ sends $\zeta_7$ to $1,$ $\phi^{\sigma}$ sends $\sigma(\zeta_7)$ to $1.$ 
\par Note that this cocycle is trivial on the subgroup fixing $K = \mathbb{Q}(\zeta_7)^+.$ Furthermore, under the notation of \cite[\S 5.3]{pss}, we see that the automorphism $\tau$ of $K$ which sends $\zeta_7 + \zeta_7^{-1}$ to $\zeta_7^4 + \zeta_7^3$ gets identified with $h = \begin{pmatrix} 2 & 0 \\ 0 & 4 \end{pmatrix},$ since our cocycle sends $\tau$ to $2 \in \mathbb{F}_7^{\times}.$ In particular, this is the desired $\tau$ automorphism per the notation of the paper \cite{pss}.
\par We now compute the relevant cohomology classes. Using the computations in \cite[\S 5.5]{pss} and the ramification conditions, these classes correspond to elements of $(\mathcal{O}_K[1/7]^{\times}/(\mathcal{O}_K[1/7]^{\times})^{7})^{\tau = 2}.$ Note that $7$ is totally ramified in $K.$ Therefore, we see that by Dirichlet's $S$-unit theorem that $\mathcal{O}_K[1/7]^{\times}$ is the product of $\mathbb{Z}/2$ (coming from $\pm 1$) and a free group of rank $3.$ Furthermore, by considering the valuations of the elements, and noting that $\tau$ preserves the valuation above $7,$ we find that this has to come from $\mathcal{O}_K^{\times},$ which is the product of a free abelian group of rank $2$ and $\mathbb{Z}/2$ (coming from $\pm 1$).  
\par We now compute the $\tau = 2$ subspace of $\mathcal{O}_K^{\times}/(\mathcal{O}_K^\times)^7.$ We first find a basis for $\mathcal{O}_K^{\times};$ there is a basis given in \cite[\href{https://www.lmfdb.org/NumberField/3.3.49.1}{Number Field 3.3.49.1}]{lmfdb}, but we use a slightly different basis for notational ease. If $u = \zeta_7 + \zeta_7^{-1},$ then it is a root of the polynomial $t^3 + t^2 - 2t - 1,$ then our choice of basis is given by $u, u+1;$ their basis is $u^2 - 1, u^2 - 2.$ We note that these are related by $u^2 - 1 = -u(u^2- 2),$ and $u^2 - 2 = -(u+1)^{-1},$ so our choice is also a basis.
\par Then, we know that $\tau(u) = -u^2 - u + 1 = -u^{-1}(u+1)$ and $\tau(u+1) = -u^{-1}.$ Therefore, the $\tau = 2$ space inside of $\mathcal{O}_K^{\times}/(\mathcal{O}_K^\times)^7$ is spanned by the image of $u(u+1)^4,$ which we can then use to construct the appropriate curves as in \cite[\S 5.4]{pss}. Evaluating the formulas found there (using Sage) gives us the explicit quartics cutting out these curves in $\mathbb{P}^2.$ Note that if our desired elliptic curve $E$ were to exist, then $X_E$ would be one of these twists, and $E$ would correspond to some rational point on one of these whose image also is the image of an integral point along $S \rightarrow \mathbb{P}^1.$  
\subsection{The Local Test}
Now that we have the appropriate twists, given as planar quartic curves, we apply the idea behind the local test of \cite[\S 7.4]{pss}. In this case, suppose that $f$ is the quartic cutting out the curve $X_E$ inside of $\mathbb{P}^2.$ Then, by \cite[Lemma 7.7]{pss}, the map $X_E \rightarrow \mathbb{P}^1 \simeq X(1)$ is given by sending $[x:y:z]$ to $\Psi_{14}(f)^3/\Psi_0(f)\Psi_6(f)^7.$ Here, $\Psi_0, \Psi_6, \Psi_{14}$ are specific homogeneous polynomials depending on the quartic equation $f$ (which are explicitly given in \cite[\S 7.1]{pss}). 
\par We can now evaluate these polynomials, just as in \cite[\S 7.4]{pss}, on residue disks inside of $\mathbb{P}^2(\mathbb{Q}_p)$ for an appropriate choice of $p.$ For sufficiently small residue disks, the polynomials $\Psi_0, \Psi_6, \Psi_{14}$ will have constant $p$-adic valuation (by taking an appropriate expansion, assuming the polynomials do not attain zeroes), and so the $j$-invariant will have constant $p$-adic valuation on these small disks. The only time that these polynomials do not have constant valuation is when one of them is zero, which corresponds to $j = 0$ or $j = \infty,$ and so our local test will capture all $j$ invariants besides possibly these two (which could only happen if those actually are images of rational points on $X_E$).
\par On the other hand, from Proposition \ref{prop: seventh power discriminant equation}, we know that the $j$-invariant of the image of any point $S(\mathbb{Z}_p) \rightarrow \mathbb{P}(1)$ will have either a nonpositive $p$-adic valuation divisible by $7$ (if $C$ is divisible by $p$), or a nonnegative $p$-adic valuation divisible by $3$ (if $B$ is divisible by $p$).
\par An explicit computation using Sage gives us the $p$-adic valuation of the $j$-invariant under the map $X_E \rightarrow X(1)$ for all but one of the allowable curves $X_E,$ and for those six curves the $p$-adic valuation of the $j$-invariants of $\mathbb{Q}_p$-points under $X_E \rightarrow X(1)$ do not have the necessary valuation when $p = 7.$ In particular, $j = 0, \infty$ do not show up for six of the curves.
\par For the last curve, we claim this is actually isomorphic to $X(7)$ as a curve over $\mathbb{Q}$ (although this is witnessed by an automorphism which is nontrivial). To check this, we claim that the following gives us the desired isomorphism from $X(7)$ to our curve $X':$ \[x \mapsto -2x/7 + 3y/7 + 6z/7, y \mapsto 2x/7 - 3y/7 + z/7
, z \mapsto 3x/7 - y/7 - 2z/7.\]
In our code, we have verification that this isomorphism over $\mathbb{Q}$ actually gives us an isomorphism between the curve we obtain and $X(7).$ 
\par However, we know all of the rational points of $X(7)$: these are the permutations of $[1:0:0]$ (see \cite[\S 3.1]{elkies}, for instance). Therefore, we know all of the rational points of this final curve. Plugging these three points into this final twist, which we explicitly do in our code, shows that $\Psi_6$ is zero, meaning that all three points correspond to $j$-invariant $\infty.$ In particular, none of them correspond to actual elliptic curves.
\par This means that none of these quartics can contribute a rational point on $X(1)$ which is the image of an integral point of $S$ where $C \neq 0,$ a contradiction to the existence of the curve $E$ we stipulated before. Hence, this curve $E$ cannot exist, so $\mathbb{Z}/7\mathbb{Z}$ is not possible as the \'{e}tale fundamental group of the N\'{e}ron model of an elliptic curve over $\mathbb{Q}$ with semistable reduction everywhere, which is what we wanted.
\subsection{Examples for $N = 1, 2, 3, 5$}
All that is left to show is that each of the four remaining values of $N$ are attained. By Lemma \ref{lem: component grows under finite etale} and Proposition \ref{prop: sufficiency for finite etale}, to give examples of isogenies $E' \rightarrow E$ whose induced morphism of N\'{e}ron models is finite \'{e}tale, it suffices to show that the Faltings height satisfies the appropriate difference, and that the discriminant of $E'$ is an appropriate power of that of $E.$
\par With this proposition, we now can assert that the following morphisms are examples for $\mathbb{Z}/2, \mathbb{Z}/3, \mathbb{Z}/5,$ proving that these groups exist as \'{e}tale fundamental groups. We give these as LMFDB labels, and the fact that the properties in the statement of Proposition \ref{prop: sufficiency for finite etale} are satisfied follows from checking the database in \cite{lmfdb} (and the equations for the curves arise from there as well).
\begin{enumerate}
    \item[$\mathbb{Z}/2$]: 17.a2 as a cover of 17.a1 (the latter having equation $y^2 + xy + y = x^3 - x^2 - 91x - 310$)
    \item[$\mathbb{Z}/3$]: 19.a2 as a cover of 19.a1 (the latter having equation $y^2 + y = x^3 + x^2 - 769x - 8470$)
    \item[$\mathbb{Z}/5$]: 11.a2 as a cover of 11.a1 (the latter having equation $y^2 + y = x^3 - x^2 - 7820x - 263580$)
\end{enumerate}
Note that for examples of trivial \'{e}tale fundamental group, each of these three covers are examples, since by what we've shown already, we cannot have any finite \'{e}tale covers of N\'{e}ron models which are degree not in the list $\{1, 2, 3, 5\}$ (and a finite \'{e}tale cover of the N\'{e}ron model of one of 17.a2, 19.a2, 11.a2 would produce a finite \'{e}tale cover of the N\'{e}ron model of its corresponding downstairs curve whose degree is outside this list). 
\par This finishes the existence part, and hence the entire proof, of the semistable reduction everywhere case for Theorem \ref{thm: classification for ECs over Q}.
\subsection{The Additive Reduction Case}
We now turn to the case where there exists a place of additive reduction. We know from Proposition \ref{prop: strong bounds for additive reduction case} that $\piet(\mathcal{E})$ has size at most $4.$ Thus, to prove Theorem \ref{thm: classification for ECs over Q}, we first show that $|\piet(\mathcal{E})| \neq 4,$ before proving that all the other cases are attained, which again we do with examples from \cite{lmfdb}. 
\par Suppose for the sake of contradiction that there existed an elliptic curve $E/\mathbb{Q}$ with a prime of additive reduction and whose N\'{e}ron model $\mathcal{E}/\mathbb{Z}$ is such that $|\piet(\mathcal{E})| = 4.$ Let $p$ be a prime of additive reduction for $E.$ Then, we have some isogeny $E' \rightarrow E$ of elliptic curves over $\mathbb{Q}$ of degree $4$ whose corresponding morphism of N\'{e}ron models is finite \'{e}tale (again, see Lemma \ref{lem: generic fiber finite etale covers are abelian} and Proposition \ref{prop: finite etale covers are neron models}). 
\par Furthermore, using the Kodaira-N\'{e}ron classification and Tate's algorithm (see \cite[Chapter IV, \S 9]{advAEC}), we see that the $p$-adic valuation of $\Delta_{E/\mathbb{Q}}$ is either equal to $f + 8$ or $f,$ where $f$ is the $p$-adic valuation of the conductor of $E,$ and the $p$-adic valuation of $\Delta_{E'/\mathbb{Q}}$ is $f + n + 4$ for some $n \geq 0.$ Recall that the conductor is an isogeny invariant (for instance, see \cite[Definition on p. 380]{advAEC}, noting that for sufficiently large $l$ that the Galois actions on $V_l(A), E[l]$ are invariant under a given isogeny).
\par Furthermore, we see that $n$ must be odd. Otherwise, the classification tells us that the component group over $p$ (after base changing to $\overline{\mathbb{F}_p}$) is $\mathbb{Z}/2 \times \mathbb{Z}/2.$ Therefore, from the proof of Proposition \ref{prop: strong bounds for additive reduction case}, the pre-image of the identity point of $\mathcal{E}_{\overline{\mathbb{F}_p}}$ under the map $\mathcal{E}' \rightarrow \mathcal{E}$ is a group variety whose $\overline{\mathbb{F}_p}$ points form the group $\mathbb{Z}/2 \times \mathbb{Z}/2.$ This also means the kernel of $E' \rightarrow E$ as a group is $\mathbb{Z}/2 \times \mathbb{Z}/2,$ and so must be $E'[2],$ which contradicts the finite \'{e}tale assumption (see Lemma \ref{lem: multiplication by N not etale}). 
\par In fact, we can be a little more precise about the reduction type of $E$ at $p.$
\begin{lemma}\label{lem: hypothetical degree 4 additive reduction specification}
In the above setting, the $p$-adic valuation of $\Delta_{E/\mathbb{Q}}$ must be $f + 8.$
\end{lemma}
\begin{proof}
Suppose it was $f$ instead. Let $\mathcal{C}$ be the minimal regular model for $E$ over $\mathbb{Z}_p,$ and let $\mathcal{E}_{\mathbb{Z}_p}, \mathcal{E}'_{\mathbb{Z}_p}$ be the base changes of the N\'{e}ron models to $\mathbb{Z}_p.$ Then, $\mathcal{C} - \mathcal{E}_{\mathbb{Z}_p}$ would consist of a single point. Then, \cite[\href{https://stacks.math.columbia.edu/tag/0BMA}{Tag 0BMA}]{stacks-project} and  \cite[\href{https://stacks.math.columbia.edu/tag/0EY7}{Tag 0EY7}]{stacks-project} tell us that $\mathcal{E}_{\mathbb{Z}_p}' \rightarrow \mathcal{E}_{\mathbb{Z}_p}$ is the restriction of some finite \'{e}tale cover $\mathcal{C}' \rightarrow \mathcal{C}.$ Note that since $\mathcal{C} - \mathcal{E}_{\mathbb{Z}_p}$ only consists of closed points, the same holds for $\mathcal{C}' - \mathcal{E}'_{\mathbb{Z}_p};$ in particular, the special fiber of $\mathcal{C}'$ only has four geometric components. 
\par This means that $\mathcal{C}'$ is a proper regular model, using \cite[\href{https://stacks.math.columbia.edu/tag/036D}{Tag 036D}]{stacks-project}. But then if $\mathcal{X}'$ is the minimal regular model for $\mathcal{E}',$ we have a morphism $\mathcal{C}' \rightarrow \mathcal{X}'$ given by contracting curves, per \cite[\href{https://stacks.math.columbia.edu/tag/0C9Z}{Tag 0C9Z}]{stacks-project}. But this is evidently impossible since $\mathcal{C}'$ has fewer components in the geometric special fiber than $\mathcal{X}'.$
\end{proof}
In other words, $E$ must have reduction type $II^*$ and $E'$ must have reduction type $I_n^*$ for $n$ odd.
\par Now, we claim that if $p$ is a place of additive reduction, then $p = 2, 3.$ To see this, suppose that $p \geq 5.$ Then, we know from \cite[Chapter IV, \S 9]{advAEC} that the $j$-invariant of $E$ is divisible by $p,$ but the $j$-invariant of $E'$ has negative $p$-adic valuation. Since $E, E'$ are isogenous, this cannot be the case; for instance, $E$ has potential good reduction but $E'$ has potential semistable reduction (e.g. see \cite[Chapter VII, Proposition 5.5]{AEC}), which cannot be the case as $E, E'$ are isogeneous. 
\par Now, we apply the same argument as in Proposition \ref{prop: finite etale covers are prime}, using Silverman's formula (\cite[Proposition 1.1]{silvermanheights}) and the estimate for $\Delta(\tau).$ We know that the kernel of $E' \rightarrow E$ consists of rational points, so we can factor this into two degree $2$ isogenies $E' \rightarrow E'' \rightarrow E$ of elliptic curves over $\mathbb{Q}.$ Note that the $p$-adic valuation of $E''$ at every prime $p$ of additive reduction must then be $f + 7,$ using the same proof as Lemma \ref{lem: hypothetical degree 4 additive reduction specification}, (replacing $E$ with $E'$).
\par Let $N$ be the product of the primes over which $E$ (and thus $E', E''$) have additive reduction; note that $1 \leq N \leq 6.$ Let $\Delta_{E/\mathbb{Q}, ss}$ be the product of the prime powers in $\Delta_{E/\mathbb{Q}}$'s factorization over all the other primes. Then, Lemma \ref{lem: component grows under finite etale} tells us that $\Delta_{E'/\mathbb{Q}, ss} = \Delta_{E''/\mathbb{Q}, ss}^2 = \Delta_{E/\mathbb{Q}, ss}^4.$ Finally, let $\tau$ be the element in the standard fundamental domain of $\mathbb{H}$ such that $E_{\mathbb{C}} \simeq \mathbb{C}/(\mathbb{Z} + \mathbb{Z} \tau),$ and let $\tau', \tau''$ be the corresponding elements for $E', E''.$ 
\par We apply Silverman's formula to find that 
\[-\frac{1}{2}\log 2 = \frac{1}{12} \log |\Delta_{E/\mathbb{Q}, ss}| - \frac{1}{12} \log N + \frac{c_{\tau} - c_{\tau''}}{12} - \frac{\pi}{6} \text{Im} \tau + \frac{\pi}{6} \text{Im} \tau'' + \frac{1}{2}\log |\frac{\text{Im} \tau}{\text{Im} \tau''}|,\] where $c_{\tau}, c_{\tau''}$ are unknown constants that have absolute value at most $1/9.$ But with $N \leq 6$ and $\text{Im} \tau \geq \frac{\sqrt{3}}{2},$ we see that $\text{Im} \tau'' = \frac{1}{2} \text{Im} \tau.$ Therefore, we see that \[|\frac{\pi}{2} \text{Im} \tau - \frac{1}{2} \log |\Delta_{E/\mathbb{Q}, ss}| + \frac{1}{2} \log N - 6 \log 2| \leq 1/9.\] Similarly, we find that \[|\frac{\pi}{4} \text{Im} \tau - \log |\Delta_{E/\mathbb{Q}, ss}| + \frac{3 - n}{2} \log N - 6 \log 2| \leq \frac{1}{9}.\] 
\par First, suppose that $n \geq 3.$ Then, taking the difference of the expressions in the absolute value, we find that \[|\frac{\pi}{4} \text{Im} \tau +  \frac{1}{2}\log |\Delta_{E/\mathbb{Q}, ss}| + \frac{n-2}{2} \log N| \leq \frac{2}{9}.\] Hence, we would need $\frac{1}{2} \log |\Delta_{E/\mathbb{Q}, ss}| \leq 2/9$ and $\frac{1}{2} \log N \leq 2/9;$ which would imply that $E$ has good reduction everywhere, which is not possible (e.g. see \cite[Exercise 8.15]{AEC}).
\par Otherwise, $n = 1.$ Then, we know that \[|\frac{\pi}{4} \text{Im} \tau + \frac{1}{2} \log |\Delta_{E/\mathbb{Q}, ss}| - \frac{1}{2} \log N | \leq 2/9,\] and so therefore \[\frac{3}{2} \log |\Delta_{E/\mathbb{Q}, ss}| - \frac{3}{2} \log N \leq -6 \log 2 + \frac{1}{3}.\] With $N \leq 6,$ we therefore have that \[\log |\Delta_{E/\mathbb{Q}, ss}| \leq \log 3/8 + \frac{2}{9}.\] or $|\Delta_{E/\mathbb{Q}, ss}| < 1,$ so this case can't happen.
\par Therefore, there is no pair of isogeneous elliptic curves with a degree $4$ isogeny between them with the required Faltings height difference, where the reduction types of the curves at each place of additive reduction are the required types, as mentioned above. This proves that $4$ is not attained as the size of the \'{e}tale fundamental group of a N\'{e}ron model.
\par To show that $1, 2, 3$ are attained, we provide the following examples from \cite{lmfdb}.
\begin{enumerate}
    \item[$\mathbb{Z}/2$]: 24.a3 as a cover of 24.a1 (the latter having equation $y^2 = x^3 - x^2 - 384x - 2772$)
    \item[$\mathbb{Z}/3$]: 27.a3 as a cover of 27.a1 (the latter having equation $y^2 + y = x^3 - 270x -1708$)
\end{enumerate}
These can be checked by checking the appropriate difference in Faltings heights and the number of components in the N\'{e}ron model over each place of bad reduction, using Proposition \ref{prop: sufficiency for finite etale}. For trivial \'{e}tale fundamental group, each of these covers are examples (since a further cover would imply the base has \'{e}tale fundamental group larger than $3,$ which we already know cannot happen). This finishes the proof of Theorem \ref{thm: classification for ECs over Q}.
\printbibliography
\end{document}